\documentclass[11pt,leqno,oneside,a4paper]{amsart}
\usepackage{amsmath,amssymb,amsthm}
\usepackage[alphabetic,abbrev]{amsrefs} 
\usepackage{tikz-cd}
\usepackage{mathrsfs} 

\makeatletter 
\def\section{\@startsection{section}{1}%
\z@{.7\linespacing\@plus\linespacing}{.5\linespacing}%
{\normalfont\bfseries\centering\large}}
\makeatother

\newcommand{\Z}{\mathbb{Z}}

\newcommand{\C}{\mathbb{C}}
\newcommand{\E}{\mathbb{E}}
\newcommand{\Mat}{\operatorname{Mat}}
\newcommand{\End}{\operatorname{End}}

\newcommand{\Orb}{\mathcal{O}}

\newcommand{\GL}{\operatorname{GL}}
\newcommand{\Sym}{\mathfrak{S}}
\newcommand{\GG}{\mathcal{G}}
\newcommand{\HH}{\mathcal{H}}
\newcommand{\Ptn}{\mathcal{P}}
\newcommand{\vt}{\operatorname{vt}}
\newcommand{\wt}{\operatorname{wt}}
\newcommand{\ov}{\overline}

\newcommand{\bi}{{\boldsymbol{i}}}  
\newcommand{\bj}{{\boldsymbol{j}}}
\newcommand{\bk}{{\boldsymbol{k}}}
\newcommand{\bl}{{\boldsymbol{l}}}
\newcommand{\bp}{{\boldsymbol{p}}}
\newcommand{\bq}{{\boldsymbol{q}}}

\newcommand{\qand}{\quad\text{and}\quad}
\newcommand{\cm}{\checkmark}
\newcommand{\op}{\operatorname{op}}
\newcommand{\Lin}{{\mathscr{L}}}

\newcommand{\bA}{\mathbf{A}}
\newcommand{\bB}{\mathbf{B}}
\newcommand{\bC}{\mathbf{C}}

\newcommand{\bM}{\mathbf{M}}
\newcommand{\bP}{\mathbf{P}}
\newcommand{\bJ}{\mathbf{J}}
\newcommand{\bQ}{\mathbf{Q}}

\newcommand{\bI}{\mathbf{I}}
\newcommand{\bG}{\mathbf{G}}

\newcommand{\bzero}{\mathbf{0}}
\newcommand{\V}{\mathbf{V}} 
\newcommand{\bv}{\mathbf{v}} 
\newcommand{\half}{\, 1/2} 
\newcommand{\fhalf}{\tfrac{1}{2}} 
\newcommand{\rt}[1]{\rotatebox{90}{#1}}

\newtheorem{thm}{Theorem}[section] 
\newtheorem{lem}[thm]{Lemma}
\newtheorem{prop}[thm]{Proposition}
\newtheorem{cor}[thm]{Corollary}

\newtheorem*{thm*}{Theorem}
\newtheorem*{lem*}{Lemma}
\newtheorem*{prop*}{Proposition}
\newtheorem*{cor*}{Corollary}
\newtheorem*{conj*}{Conjecture}
\newtheorem{property}{Property}

\newtheorem*{property*}{Property}
\newtheorem*{properties*}{Properties}

\theoremstyle{definition}
\newtheorem{defn}[thm]{Definition}
\newtheorem{example}[thm]{Example}

\newtheorem*{defn*}{Definition}
\newtheorem*{example*}{Example}
\newtheorem*{examples*}{Examples}
\newtheorem*{meth*}{Method}

\newtheorem{rmk}[thm]{Remark}

\newtheorem{algorithm}[thm]{Algorithm}
\newtheorem*{rmk*}{Remark}
\newtheorem*{rmks*}{Remarks}
\newtheorem*{algorithm*}{Algorithm}

\allowdisplaybreaks

\renewcommand{\theenumi}{\alph{enumi}} 

\begin{document}
\title[Schur--Weyl duality for partition algebras]%
      {Integral Schur--Weyl duality\\for partition algebras}

  \author{Chris Bowman}
  \address{School of Mathematics, Statistics, and
  Actuarial Science, University of Kent, Canterbury, Kent, 
  CT2 7NZ, UK}
  \email{C.D.Bowman@kent.ac.uk}

  \author{Stephen Doty}
  \address{Department of Mathematics and Statistics, Loyola University
    Chicago, Chicago, IL 60660 USA}
  \email{doty@math.luc.edu}

  \author{Stuart Martin}
  \address{DPMMS, Centre for Mathematical Sciences, Wilberforce Road,
  Cambridge, CB3 0WB, UK}
  \email{S.Martin@dpmms.cam.ac.uk}


\begin{abstract}\noindent
Let $\V$ be a free module of rank $n$ over a commutative ring $\Bbbk$.
We prove that tensor space $\V^{\otimes r}$ satisfies Schur--Weyl
duality, regarded as a bimodule for the action of the group algebra of
the Weyl group of $\GL(\V)$ and the partition algebra $\Ptn_r(n)$ over
$\Bbbk$.  We also prove a similar result for the half partition
algebra.
\end{abstract}
\maketitle

\section*{Introduction}\noindent
A number of instances of Schur--Weyl duality (a bimodule for which the
centraliser of each action equals the image of the other) have been
established in positive characteristic over the past forty years,
including: \cite{Benson-Doty} (extending \cites{Carter-Lusztig, dCP,
  Green:book}) for general linear and symmetric groups; \cite{DDH} for
symplectic groups and Brauer algebras; \cite{DH} for orthogonal groups
and Brauer algebras (characteristic not 2); \cite{Garge-Nebhani} for
special orthogonal groups and the Brauer--Grood algebra; \cites{DDS1,
  DDS2, DD, Donkin} for general linear groups and walled-Brauer
algebras. In all these cases, semisimple versions of Schur--Weyl
duality were observed much earlier as an application of
Artin--Wedderburn theory, and the extension to positive characteristic
(where representations tend to be non-semisimple) is much more
difficult to establish. This paper continues the above body of work, by
extending the Schur--Weyl duality between symmetric groups and
partition algebras to non-semisimple situations.

Let $\Bbbk$ be a commutative ring (always with 1). Fix a free
$\Bbbk$-module $\V$ of rank $n$, and fix a $\Bbbk$-basis $\{\bv_1,
\dots, \bv_n\}$ of $\V$. As explained in Section \ref{sec:bimodules},
tensor space $\V^{\otimes r}$ is a $(\Bbbk W_n, \Ptn_r(n))$-bimodule,
where $W_n$ is the Weyl group of $\GL(\V)$. We identify $\V^{\otimes
  r}$ with $\V^{\otimes r} \otimes \bv_n \subset \V^{\otimes (r+1)}$,
which becomes a $(\Bbbk W_{n-1}, \Ptn_{r+\half}(n))$-bimodule by
restriction.

The purpose of this paper is to show that Schur--Weyl duality holds
for both bimodule structures on $\V^{\otimes r}$; that is, we have
the following result.

\begin{thm*}[Schur--Weyl duality]
  Let $\Bbbk$ be a commutative ring. Then:
  \begin{enumerate}
  \item The centraliser algebras $\End_{\Ptn_r(n)}(\V^{\otimes r})$,
  $\End_{W_n}(\V^{\otimes r})$ coincide with the images of the
  representations $\Bbbk W_n \to \End_\Bbbk(\V^{\otimes r})$,
  $\Ptn_r(n)^{\op} \to \End_\Bbbk(\V^{\otimes r})$, respectively.
  
  \item Similarly, the centraliser algebras
    $\End_{\Ptn_{r+\half}(n)}(\V^{\otimes r})$,
  $\End_{W_{n-1}}(\V^{\otimes r})$ coincide with the images of the
  representations $\Bbbk W_{n-1} \to \End_\Bbbk(\V^{\otimes r})$,
  $\Ptn_{r+\half}(n)^{\op} \to \End_\Bbbk(\V^{\otimes r})$,
  respectively.
  \end{enumerate}
\end{thm*}

\noindent
This result is well known in case $\Bbbk=\C$; it follows from standard
facts in the theory of semisimple algebras (see e.g.,
\cite{HR-2005}*{Thms~5.4, 3.6}). Our contribution is to extend the
result to arbitrary commutative rings $\Bbbk$.
%

As an application of the main theorem, we prove
\cite{BDM:second}*{Corollary 7.5} that the centraliser algebras
$\End_{\Ptn_r(n)}(\V^{\otimes r})$,
$\End_{\Ptn_{r+\half}(n)}(\V^{\otimes r})$ are cellular algebras over
a commutative ring, in the sense of \cite{Graham-Lehrer}.  Recently
\cite{Donkin:Cellularity} proved a similar result for
$\End_{W_n}(\V^{\otimes r})$; see also \cites{BH:1, BH:2} for related
results over fields of characteristic zero.

The proof of the theorem (for general $\Bbbk$) is obtained as
follows. First, the fact that the images of the maps from the
partition algebras coincide with the $W_n$ and $W_{n-1}$ centralisers
already appears in the proof of \cite{HR-2005}*{Thm.~3.6}; the
combinatorial argument there works for any commutative ring
$\Bbbk$. So we only need to consider the representations of $\Bbbk
W_n$ and $\Bbbk W_{n-1}$. The proof that those representations surject
onto the appropriate centraliser algebras is Theorem \ref{thm:main1}
of this paper, our main result. In particular, the centraliser
algebras $\End_{\Ptn_r(n)}(\V^{\otimes r})$,
$\End_{\Ptn_{r+1/2}(n)}(\V^{\otimes r})$ are spanned over $\Bbbk$ by
elements of the form $\bP(w)^{\otimes r}$, where $w \in W_n$,
$W_{n-1}$ respectively. Here, $\bP(w)$ is the permutation matrix
corresponding to $w$.

\subsection*{Acknowledgments.}
Part of the work on this paper was carried out in December 2017 at the
Institute for Mathematical Sciences, National University of
Singapore. We are grateful to the organisers of the conference
``Representation theory of symmetric groups and related algebras'' at
that venue, for providing an opportunity for collaboration.

\section{Tensor space}\label{sec:bimodules}\noindent
As above, $\Bbbk$ is a fixed commutative ring (with unit). We denote
its zero by $0_\Bbbk$ and its unit by $1_\Bbbk$. We identify ordinary
integers $m \in \Z$ with elements of $\Bbbk$ by means of the canonical
map $\Z \to \Bbbk$ defined by $m \mapsto m 1_\Bbbk$. In particular,
this identifies $0,1 \in \Z$ respectively with $0_\Bbbk, 1_\Bbbk \in
\Bbbk$.

Throughout this paper, $\V$ denotes a fixed free $\Bbbk$-module of
rank $n$ with a given basis $\{\bv_1, \dots, \bv_n\}$, by means of which
we identify $\V$ with $\Bbbk^n$. For any positive integer $r$, the set
\begin{equation}\label{eq:TS-basis}
\{ \bv_{i_1} \otimes \cdots \otimes \bv_{i_r} : i_1,\dots, i_r = 1,
\dots, n \}
\end{equation}
is a basis of the $r$th tensor power $\V^{\otimes r}$. The general
linear group $\GL(\V)$ of $\Bbbk$-linear automorphisms of $\V$ acts
naturally on the left on $\V$; this action extends diagonally to an
action on $\V^{\otimes r}$. The symmetric group $\Sym_r$ acts on the
right on $\V^{\otimes r}$ by permuting the tensor positions; this
action is known as the \emph{place-permutation} action, defined by
\begin{equation}\label{eq:PP-action}
(\bv_{i_1} \otimes \cdots \otimes \bv_{i_r})^\sigma = \bv_{i_{1 \sigma^{-1}}}
\otimes \cdots \otimes \bv_{i_{r \sigma^{-1}}} \, , \text{ for } \sigma \in
\Sym_r
\end{equation}
extended linearly. (We write maps in $\Sym_r$ on the \emph{right} of
their arguments.) Thus we have commuting actions of the groups
$\GL(\V)$, $\Sym_r$ on $\V^{\otimes r}$, making $\V^{\otimes r}$ into
a $(\Bbbk \GL(\V), \Bbbk \Sym_r)$-bimodule.  Classical Schur--Weyl
duality is the statement that the centraliser of each action is
generated by the image of the other. It was proved originally over
$\Bbbk = \C$ by Schur, extended to infinite fields by J.~A.~Green and
many other authors, and extended to sufficiently large fields in
\cite{Benson-Doty}.

Let $W_n$ be the Weyl group of $\GL(\V)$, i.e., the group of elements
of $\GL(\V)$ permuting the basis $\{\bv_1, \dots, \bv_n\}$.  We identify
$W_n$ with the group of permutation matrices, regarded as matrices
with entries from $\Bbbk$. By restricting the action of $\GL(\V)$ to
$W_n$, we obtain left actions of $W_n$ on $\V$ and $\V^{\otimes
  r}$. To be explicit, $w \in W_n$ acts by
\begin{equation}\label{eq:W_n-action}
  w (\bv_{j_1} \otimes \cdots \otimes \bv_{j_r}) = \bv_{w(j_1)} \otimes
  \cdots \otimes \bv_{w(j_r)}.
\end{equation}
(We write maps in $W_n$ on the \emph{left} of their arguments.)
Extended linearly, the action of $W_n$ defines a linear representation
$\Bbbk W_n \to \End_\Bbbk( \V^{\otimes r} )$ of the group algebra
$\Bbbk W_n$.

Of course, $W_n \cong \Sym_n^{\op}$, the opposite group. We distinguish
these symmetric groups notationally throughout this paper, because
their actions on tensor space are very different even when $n=r$.

Let $\Ptn_r(\delta) \supset \Sym_r$ be the partition algebra
(introduced independently by Paul P.~Martin \cites{Marbook,
  Martin:1994} and V.~F.~R.~Jones \cite{Jones} in relation to the Potts
model in particle physics) on $2r$ vertices, with parameter $\delta
\in \Bbbk$.  The algebra $\Ptn_r(\delta)$ has a $\Bbbk$-basis in
bijection with the collection of set partitions on
\[ \{1, \dots, r\} \cup \{1', \dots, r'\}. \]
Each basis element $d$ may be regarded as a graph with $2r$ vertices
arranged in two rows, with vertices numbered $1, \dots, r$ along the
top and $1', \dots r'$ along the bottom. Two vertices in the graph are
connected by an edge if and only if they lie in the same subset of the
set partition $d$. Multiplication in the algebra may be defined on the
basis elements by stacking diagrams and removing any connected
components that contain no vertices from the top and bottom rows of
the stack. After removing such interior components, the result of
stacking $d_1$ above $d_2$ is a new diagram $d_3$, and the
multiplication is defined by
\begin{equation}\label{eq:Ptn-alg-mult}
  d_1 d_2 = \delta^k \, d_3
\end{equation}
where $k$ is the number of removed interior connected components. It
can be checked that this rule, extended linearly, defines an
associative multiplication on $\Ptn_r(\delta)$.

%
\begin{example} The following 3 diagrams all depict the same set-partition:  $\big\{\{1,3,\bar3,\bar4\},\{2,\bar1\},\{4\},\{5,\bar2,\bar5\}\big\}$.
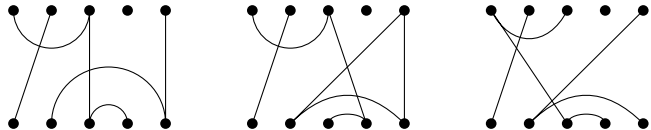
$$		\begin{tikzpicture}[scale=0.5]
			\foreach \x in {1,...,5}
				{\fill[black] (\x,3) circle (4pt);
				\fill[black] (\x,0) circle (4pt);}
				\draw (1,3) arc (-180:0:1) -- (3,0) arc (180:0:0.5);
				\draw (2,3)--(1,0);
				\draw (2,0) arc (180:0:1.5) -- (5,3);
		\end{tikzpicture}\hspace{1cm}
		\begin{tikzpicture}[scale=0.5]
			\foreach \x in {1,...,5}
				{\fill[black] (\x,3) circle (4pt);
				\fill[black] (\x,0) circle (4pt);}
				\draw (1,3) arc (-180:0:1) -- (4,0) arc (0:180:0.5 and 0.25);
				\draw (2,3)--(1,0);
				\draw (2,0) .. controls (3,1) and (4,1) .. (5,0) -- (5,3) -- (2,0);
		\end{tikzpicture}\hspace{1cm}
		\begin{tikzpicture}[scale=0.5]
			\foreach \x in {1,...,5}
				{\fill[black] (\x,3) circle (4pt);
				\fill[black] (\x,0) circle (4pt);}
				\draw (3,3) .. controls (2.5,2) and (1.5,2) .. (1,3) -- (3,0) arc (180:0:0.5 and 0.25);
				\draw (2,3)--(1,0);
				\draw (5,0) .. controls (4,1) and (3,1) .. (2,0) -- (5,3);
		\end{tikzpicture}$$
		\end{example}
%

	\begin{example}
In the following example,  the 
diagram on the right-hand side of the equality
is obtained by 
stacking the two diagrams on the left of the equality on top of each other (with the leftmost diagram on top)
and removing the singleton from the middle of the diagram (at the expense of multiplication by the parameter $\delta^1$).  
	$$	
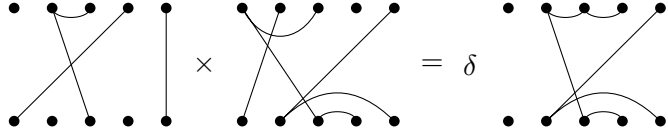
\begin{tikzpicture}[scale=0.5]
 \foreach \x in {1,...,5,7,8,...,11,14,15,...,18}
	 	{\fill[black] (\x,0) circle (4pt);
	 	\fill[black] (\x,3) circle (4pt);}
	 	
	\draw (3,3) arc (0:-180:0.5 and 0.25) -- (3,0);\draw (4,3)--(1,0);\draw (5,3)--(5,0);
	\draw (6,1.5) node {$\times$};
	\draw (9,3) .. controls (8.5,2) and (7.5,2) .. (7,3) -- (9,0) arc (180:0:0.5 and 0.25);
	\draw (8,3)--(7,0);
	\draw (11,0) .. controls (10,1) and (9,1) .. (8,0) -- (11,3);
	\draw (12,1.5) node {$=$};\draw (13,1.5) node {$\delta$};
	\draw (18,0) .. controls (17,1) and (16,1) .. (15,0) -- (18,3);
	\draw (17,3) arc (0:-180:0.5 and 0.25) arc (0:-180:0.5 and 0.25) -- (16,0) arc (180:0:0.5 and 0.25);
	\end{tikzpicture}$$
	\end{example}

The half partition algebra $\Ptn_{r+\half}(\delta)$, introduced in
\cite{Martin:2000}, is the subalgebra of $\Ptn_{r+1}(\delta)$ spanned
by diagrams such that vertex $r+1$ is connected to vertex $(r+1)'$.

By specialising the parameter $\delta$ to $n$, we obtain a linear
representation of $\Ptn_r(n)^{\op}$, defined as follows. Let $I(n,r) =
\{1, \dots, n\}^r$ be the set of multi-indices of length $r$. To
simplify the notation, we often write
\[
  \text{$i_1 \cdots i_r$ instead of $(i_1, \dots, i_r)$}
\]
for an element of $I(n,r)$.  Elements of $I(n,r)$ can also be written
as $i_{1'} \cdots i_{r'}$ or $(i_{1'}, \dots, i_{r'})$.  Connected
components of a diagram are called \emph{blocks}; they correspond to
the subsets of the underlying set partition. Following \cite{HR-2005},
we define a scalar $(d)^{i_1 \cdots i_r}_{i_{1'} \cdots i_{r'}} \in
\Bbbk$, for a diagram $d$ and any $i_1 \cdots i_r$, $i_{1'} \cdots
i_{r'}$ in $I(n,r)$, by
\[
(d)^{i_1 \cdots i_r}_{i_{1'} \cdots i_{r'}} = 
\begin{cases}
  1 & \text{if } i_\alpha = i_\beta \text{ whenever } \alpha \ne
  \beta \text{ are in the same block of } d\\ 0 & \text{otherwise}.
\end{cases}
\]
Here, the indices $\alpha,\beta$ may be primed or unprimed.  The
diagram $d$ acts on $\V^{\otimes r}$, on the right, by the rule
\begin{equation}\label{eq:Ptn-action}
(\bv_{i_1} \otimes \cdots \otimes \bv_{i_r})^d = \sum_{i_{1'} \cdots
    i_{r'}} (d)^{i_1 \cdots i_r}_{i_{1'} \cdots i_{r'}} \, (\bv_{i_{1'}}
  \otimes \cdots \otimes \bv_{i_{r'}})
\end{equation}
extended linearly.  Note that if $d \in \Sym_r$ is a permutation
diagram, then $d$ acts by the usual place-permutation action. Extended
linearly, this action defines a linear representation $\Ptn_r(n)^{\op}
\to \End_\Bbbk(\V^{\otimes r})$.  The commuting actions define a
$(\Bbbk W_n, \Ptn_r(n))$-bimodule structure on $\V^{\otimes r}$.  By
identifying $\V^{\otimes r}$ with $\V^{\otimes r} \otimes \bv_n \subset
\V^{\otimes (r+1)}$, by restriction $\V^{\otimes r}$ may also be
regarded as a $(\Bbbk W_{n-1}, \Ptn_{r+\half}(n))$-bimodule, where
$W_{n-1}$ is identified with the subgroup of $W_n$ consisting of the
permutations fixing $n$.

Because the actions commute, the representations $\Bbbk W_n
\to \End_\Bbbk(\V^{\otimes r})$, $\Ptn_r(n) \to \End_\Bbbk(\V^{\otimes
  r})$ induce $\Bbbk$-algebra homomorphisms
\begin{equation}\label{eq:induced-homs}
\begin{aligned}
  \Phi_{n,r}&: \Bbbk W_n \to \End_{\Ptn_r(n)}( \V^{\otimes r} ),\\
  \Psi_{n,r}&: \Ptn_r(n)^{\op} \to \End_{W_n}( \V^{\otimes r} )
\end{aligned}
\end{equation}
into the respective centraliser algebras. Similarly, by restriction we
have induced $\Bbbk$-algebra homomorphisms
\begin{equation}\label{eq:induced-homs-half}
\begin{aligned}
  \Phi_{n,r+\half}&: \Bbbk W_{n-1}
  \to \End_{\Ptn_{r+\half}(n)}(\V^{\otimes r} ), \\ \Psi_{n,r+\half}&:
  \Ptn_{r+\half}(n)^{\op} \to \End_{W_{n-1}}( \V^{\otimes r} ).
\end{aligned}
\end{equation}
Schur--Weyl duality is equivalent to the surjectivity of these induced
homomorphisms. As noted above, the combinatorial argument given in
\cite{HR-2005}*{Thm.~3.6}, which works over any commutative ring
$\Bbbk$, proves the surjectivity of the maps $\Psi_{n,r}$,
$\Psi_{n,r+\half}$. Indeed, the partition algebra was originally
defined with that property in mind.

Thus, we only need to prove the surjectivity of the induced maps
$\Phi_{n,r}$, $\Phi_{n,r+\half}$ defined in \eqref{eq:induced-homs},
\eqref{eq:induced-homs-half}.

\section{Generalised doubly-stochastic matrices}\noindent
\label{sec:GDS}%
We denote the entry in the $i$th row and $j$th column of a matrix
$\bM$ by $m^i_j$, and write $\bM=(m^i_j)$. In this paper, we will
always follow this convention of using bold letters for matrices and
lower case letters for their entries. 
It will be convenient to employ the following terminology; see e.g.,
\cites{Brualdi, Johnsen, Gibson, Lai}.

\begin{defn}\label{def:GDS}
  An $n \times n$ matrix $\bM = \big( m^i_j \big)_{i,j = 1, \dots, n}$
  is \emph{generalised doubly-stochastic} (GDS) if there is some $s =
  s(\bM)$ in $\Bbbk$ such that both:
  \begin{enumerate}
  \item $\sum_{j=1}^n m^i_j = s$, for all $i=1, \dots, n$, and
  \item $\sum_{i=1}^n m^i_j = s$,   for all $j = 1, \dots, n$. 
  \end{enumerate}
\end{defn}

In other words, $\bM$ is GDS if there is a common value for all its
row and column sums.  

\begin{lem}\label{lem:GDS}
  Assume that $n > 1$. An $n \times n$ matrix $\bM$ over the ring
  $\Bbbk$ is GDS if and only if $\bM$ commutes with the matrix $\bJ_n =
  (1)_{i,j=1,\dots, n}$ of all ones.
\end{lem}

\begin{proof}
The matrix $\bJ_n \bM$ is the $n \times n$ matrix with the column sum
${\rm co}_i(\bM) := \sum_j m^i_j$ in each entry of the $i$th column. On the
other hand, $\bM \bJ_n$ is the $n \times n$ matrix with the row sum
${\rm ro}_j(\bM) := \sum_i m^i_j$ in each entry of the $j$th row. So $\bM$
commutes with $\bJ_n$ if and only if ${\rm co}_i(\bM) = {\rm ro}_j(\bM)$ for all
$i,j = 1, \dots, n$. Since $n > 1$, it follows that this is so if and
only if all the row and column sums have a common value.
\end{proof}

If $\bM$ is identified with the $\Bbbk$-linear endomorphism of $\V$
defined by $\bv_j \mapsto \sum_i m^i_j \, \bv_i$, then $s(\bM)$ is an
eigenvalue for the eigenvector $\bv_1 + \cdots + \bv_n$. The proof is
easy. This observation leads to the following characterisation of GDS
operators, where GDS operators on $\V$ are defined to be linear
operators whose matrices with respect to the basis $\{\bv_1, \dots,
\bv_n\}$ are GDS.



\section{Description of the invariants $\E_\Bbbk(n,r)$}%
\label{sec:invariants}\noindent
To ease the notation, we henceforth put $\E_\Bbbk(n,r)
= \End_{\Ptn_r(n)}(\V^{\otimes r})$. Similarly, we set
$\E_\Bbbk(n,r+\fhalf) = \End_{\Ptn_{r+\half}(n)}(\V^{\otimes r})$. The
purpose of this section is to describe $\E_\Bbbk(n,r)$; a description
of $\E_\Bbbk(n,r+\fhalf)$ will be given in the next section.

We introduce some additional notation. For each multi-index $\bi =
i_1\cdots i_r$ in $I(n,r)$, $\sigma$ in $\Sym_r$, and $w$ in $W_n$, we
set
\[
  \bi^\sigma = (i_1\cdots i_r)^\sigma = i_{1 \sigma^{-1}} \cdots
  i_{r \sigma^{-1}}, \quad w \bi = w(i_1 \cdots i_r) = w(i_1) \cdots w(i_r).
\] 
The assignment $(\bi , \sigma) \mapsto \bi^\sigma$ defines a right
action (the place-permutation action) of the symmetric group $\Sym_r$
on the set $I(n,r)$; the assignment $(w,\bi) \mapsto w \bi$ defines a
left action of $W_n$ on $I(n,r)$. These left and right actions on
$I(n,r)$ commute: $(w\bi)^\sigma = w(\bi^\sigma)$, for all $w \in
W_n$, $\sigma \in \Sym_r$. Set:
\begin{equation}
  \bv_\bi = \bv_{i_1} \otimes \cdots \otimes \bv_{i_r}.
\end{equation}
Then the basis of $\V^{\otimes r}$ given in \eqref{eq:TS-basis} is
$\{\bv_\bi: \bi \in I(n,r)\}$, and the commuting actions of $W_n$ and
$\Sym_r$ on $\V^{\otimes r}$ considered in Section \ref{sec:bimodules}
are given by the rules $(w, \bv_\bi) \mapsto \bv_{w\bi}$ and
$(\bv_\bi, \sigma) \mapsto \bv_{\bi^\sigma}$.

Orbits for the left action of $W_n$ on $I(n,r)$ are called
\emph{value-types} and may be identified with set partitions of $\{1,
\dots, r\}$.  The subsets in the value-type $\vt(\bi) = \Lambda$ of
$\bi = i_1 \cdots i_r$ record the positions at which the distinct
values that appear are constant. More precisely, we make the following
definition.

\begin{defn}\label{def:value-type}
  Let $\bi = i_1 \cdots i_r \in I(n,r)$ be given.  For each $v =
  1, \dots, n$, let $\Lambda'_v$ be the set of all positions
  $\alpha = 1, \dots, r$ for which $i_\alpha = v$. Then $\{1, \dots,
  r\} = \bigcup_{v=1}^n \Lambda'_v$. Discard any empty $\Lambda'_v$ to
  obtain the set partition $\Lambda = \{ \Lambda'_v : \Lambda'_v \ne
  \emptyset\}$ which defines $\vt(\bi)$. Let $\ell(\Lambda)$ be the
  number of non-empty subsets in $\Lambda$.
\end{defn}

For example, the multi-index $\bi = abbcabc$ (for distinct elements
$a,b,c \in \{1, \dots, n\}$) has value-type $\vt(\bi) = \{\{1,5\},
\{2,3,6\}, \{4,7\}\}$.

Recall that orbits for the right action of $\Sym_r$ on $I(n,r)$ are
called \emph{weights} and may be identified with weak compositions of
$r$ of length at most $n$. To be precise, the weight $\wt(\bi)$ of a
multi-index $\bi = i_1 \cdots i_r$ is the composition $\wt(\bi) =
(\mu_1, \dots, \mu_n)$ where for each value $v = 1, \dots, n$ the
statistic $\mu_v$ counts the number of positions $\alpha = 1, \dots,
r$ such that $i_\alpha = v$. In other words, $\mu_v = |\Lambda'_v|$
for each $v = 1, \dots, n$.

We also need to consider the set $\Omega(n,r)$ of orbits for the right
action of $\Sym_r$ on $I(n,r) \times I(n,r)$ defined by
$(\bi, \bj)^\sigma = (\bi^\sigma, \bj^\sigma)$.

In order to describe the invariants $\End_{\Ptn_r(n)}(\V^{\otimes r})$
it suffices to consider a set of generators. Halverson and Ram
\cite{HR-2005} showed that $\Ptn_r(\delta)$ is generated by the
diagrams:
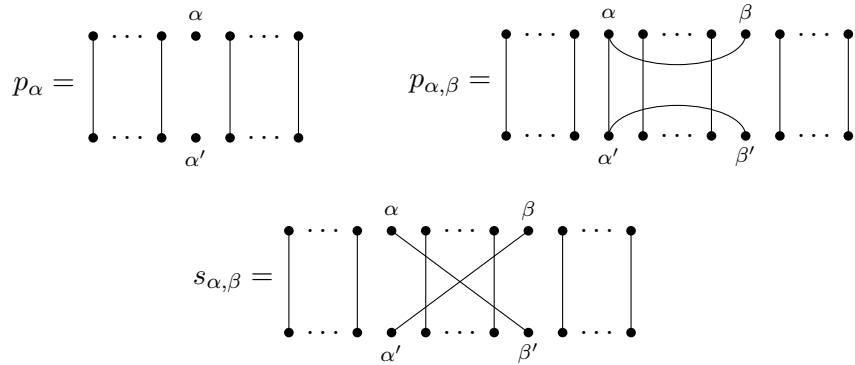
\[
\begin{tikzpicture}[scale=0.45]
	\draw (-1,1.5) node {$p_\alpha=\quad$};
	\foreach \x in {0,2,3,4,6}
	{\fill[black] (\x,0) circle (4pt);
	\fill[black] (\x,3) circle (4pt);}
	\foreach \x in {0,2,4,6}
	{\draw (\x,3)--(\x,0);}
	\draw (3,3.6) node {$\scriptstyle\alpha$};
	\draw (3,-0.6) node {$\scriptstyle\alpha'$};
\draw (1,0) node {$\ldots$};\draw (1,3) node {$\ldots$};
\draw (5,0) node {$\ldots$};\draw (5,3) node {$\ldots$};
\end{tikzpicture}
\hspace{0.5in}
\begin{tikzpicture}[scale=0.45]
	\draw (-1.2,1.5) node {$p_{\alpha,\beta}=\quad$};
	\foreach \x in {0,2,3,4, 5.5+.5,6.5+.5,8,10}
	{\fill[black] (\x,0) circle (4pt);
	\fill[black] (\x,3) circle (4pt);}
	\foreach \x in {0,2,4, 5.5+.5,7.5+.5,10}
	{\draw (\x,3)--(\x,0);}
	\draw (3,3)--(3,0);   \draw (3,3) arc (0:180:-2 and -0.9) ;
        \draw (3,0) arc (0:180:-2 and 0.9) ;
	\draw (3,3.6) node {$\scriptstyle\alpha$};\draw (7,3.6)
        node {$\scriptstyle\beta$};
	\draw (3,-0.6) node {$\scriptstyle\alpha'$};\draw (7,-0.6)
        node {$\scriptstyle\beta'$};
\draw (1,3) node {$\ldots$};\draw (1,0) node {$\ldots$};
\draw (5,3) node {$\ldots$};\draw (5,0) node {$\ldots$};
\draw (9,3) node {$\ldots$};\draw (9,0) node {$\ldots$};
\end{tikzpicture}
\]
\[
\begin{tikzpicture}[scale=0.45]
	\draw (-1.2,1.5) node {$s_{\alpha,\beta}=\quad$};
	\foreach \x in {0,2,3,4, 5.5+.5,6.5+.5,8,10}
	{\fill[black] (\x,0) circle (4pt);
	\fill[black] (\x,3) circle (4pt);}
	\foreach \x in {0,2,4, 5.5+.5,7.5+.5,10}
	{\draw (\x,3)--(\x,0);}
	\draw (3,3)--(7,0);\draw (7,3)--(3,0);
	\draw (3,3.6) node {$\scriptstyle\alpha$};\draw (7,3.6)
            node {$\scriptstyle\beta$};
	\draw (3,-0.6) node {$\scriptstyle\alpha'$};\draw (7,-0.6)
            node {$\scriptstyle\beta'$};
\draw (1,3) node {$\ldots$};\draw (1,0) node {$\ldots$};
\draw (5,3) node {$\ldots$};\draw (5,0) node {$\ldots$};
\draw (9,3) node {$\ldots$};\draw (9,0) node {$\ldots$};
\end{tikzpicture}\hspace{0.5cm}
\]
where $1 \leq \alpha<\beta \leq r$. In fact, $\Ptn_r(\delta)$ is
generated by the usual Coxeter generators $s_{\alpha,\alpha+1}$ for
$\alpha = 1, \dots, r-1$ along with just one $p_\alpha$ and one
$p_{\alpha,\beta}$.

Let $\GG_r(n)$, $\HH_r$, $\Bbbk\Sym_r$ be the subalgebras of
$\Ptn_r(n)$ respectively generated by the $p_\alpha$,
$p_{\alpha,\beta}$, and $s_{\alpha,\beta}$ pictured above. Note that
$\HH_r$ is independent of $n$. We shall separately consider the
centraliser algebras of these three subalgebras:
\begin{gather*}
G_\Bbbk(n,r) = \End_{\GG_r(n)}(\V^{\otimes r}), \quad
H_\Bbbk(n,r) = \End_{\HH_r} (\V^{\otimes r}),
\\ S_\Bbbk(n,r) = \End_{\Sym_r} (\V^{\otimes r}).
\end{gather*}
Then we have $\E_\Bbbk(n,r) = G_\Bbbk(n,r) \cap H_\Bbbk(n,r) \cap
S_\Bbbk(n,r)$. Note that $S_\Bbbk(n,r)$ is the classical Schur algebra
appearing in \cite{Green:book}. Furthermore, we have
$\HH_1=\Bbbk\Sym_1=\Bbbk$ and thus
$H_\Bbbk(n,1)=S_\Bbbk(n,1)=\End_\Bbbk(\V)$, which means that
$\E_\Bbbk(n,1) = G_\Bbbk(n,1)$. By Lemma \ref{lem:GDS}, the algebra
$G_\Bbbk(n,1)$ coincides with the set of $n \times n$ GDS matrices
considered in Section \ref{sec:GDS}.

Henceforth, we write $\Mat_{I(n,r)}(\Bbbk)$ for the set of $n^r \times
n^r$ matrices over $\Bbbk$, with rows and columns indexed by the set
$I(n,r)$. We always identify the matrix $\bA = (a^\bi_\bj)$ with the
$\Bbbk$-linear endomorphism of $\V^{\otimes r}$ defined on basis
elements by $\bv_\bj \mapsto \sum_{\bi \in I(n,r)} a^{\bi}_{\bj} \,
\bv_{\bi}$.

\begin{prop}\label{prop:GHS}
  Let $G_\Bbbk(n,r)$, $H_\Bbbk(n,r)$, and $S_\Bbbk(n,r)$ be the
  subalgebras of $\End_\Bbbk( \V^{\otimes r} )$ consisting of
  the endomorphisms commuting with the action of all the $p_\alpha$,
  $p_{\alpha,\beta}$, and $s_{\alpha,\beta}$ respectively. Let $\bA =
  ( a^{\bi}_{\bj} ) \in \Mat_{I(n,r)}(\Bbbk)$. Then
\begin{enumerate}
\item $\bA$ belongs to $G_\Bbbk(n,r)$ if and only if for each place
  $\alpha = 1, \dots, r$ and for each fixed $\bp = i_1 \cdots
  i_{\alpha-1} i_{\alpha+1}\cdots i_r$, $\bq = j_1 \cdots j_{\alpha-1}
  j_{\alpha+1} \cdots j_r$ in $I(n,r-1)$ there is some scalar
  $b^\bp_\bq(\alpha)$ in $\Bbbk$ such that
  \begin{align*}
    \sum_{i = 1}^n a^{i_1 \cdots i_{\alpha-1}\, i\, i_{\alpha+1}
      \cdots i_r}_{j_1 \cdots j_{\alpha-1}\, j\, j_{\alpha+1} \cdots
      j_r} &= b^\bp_\bq(\alpha), \text{ for any } j = 1, \dots, n
    \\ \intertext{ and} \sum_{j = 1}^n a^{i_1 \cdots i_{\alpha-1}\,
      i\, i_{\alpha+1} \cdots i_r}_{j_1 \cdots j_{\alpha-1}\, j\,
      j_{\alpha+1} \cdots j_r} &= b^\bp_\bq(\alpha), \text{ for any }
    i = 1, \dots, n.
  \end{align*}

\item $\bA$ belongs to $H_\Bbbk(n,r)$ if and only if $\bA$ preserves
  value-type, in the sense: $a^{\bi}_{\bj} = 0$ for all pairs
  $(\bi,\bj) \in I(n,r) \times I(n,r)$ such that $\vt(\bi) \ne
  \vt(\bj)$.
  
\item $\bA$ belongs to $S_\Bbbk(n,r)$ if and only if $\bA$ is constant
  on each place-permutation orbit $\Orb$ in $\Omega(n,r)$, i.e., if
  $a^\bi_\bj = a^\bk_\bl$ whenever $(\bi,\bj)$ and $(\bk,\bl)$ are in
  the same orbit $\Orb \in \Omega(n,r)$.
\end{enumerate}
\end{prop}

\begin{proof}
(a) Suppose that $1 \le \alpha \le r$. By the definition in equation
\eqref{eq:Ptn-action}, the matrix $\Psi_{n,r}(p_{\alpha})$
representing $p_{\alpha}$ with respect to the basis $\{\bv_\bi : \bi
\in I(n,r)\}$ has the form
\[
\big( \delta_{i_1,j_1} \cdots \delta_{i_{\alpha-1}, j_{\alpha-1}}
\delta_{i_{\alpha+1}, j_{\alpha+1}} \cdots \delta_{i_r, j_r} \big)_{\bi,\bj
  \in I(n,r) \times I(n,r)} \, .
\]
More succinctly, the matrix can be written as $(\bI_n)^{\otimes
  \alpha-1} \otimes \bJ_n \otimes (\bI_n)^{\otimes r-\alpha}$, where
$\bJ_n$ is the $n \times n$ matrix defined in Section
\ref{sec:GDS}. It follows immediately from Lemma \ref{lem:GDS} that
the commutant $\End_{p_\alpha}(\V^{\otimes r})$ of $p_\alpha$
is the set of endomorphisms satisfying the displayed condition in part
(a) of the proposition. Thus, the centraliser $G_\Bbbk(n,r)$ of all
the $p_\alpha$ for $1 \le \alpha \le r$ is the set of endomorphisms
satisfying the condition for all $\alpha$. This proves part (a).
  
(b) Suppose that $1 \le \alpha < \beta \le r$. By
\eqref{eq:Ptn-action}, the matrix $\Psi_{n,r}(p_{\alpha,\beta})$
representing $p_{\alpha,\beta}$ with respect to the basis $\{\bv_\bi :
\bi \in I(n,r)\}$ is
\[
\big( \delta_{i_{1},j_{1}} \cdots \delta_{i_{\alpha-1},j_{\beta-1}}
\delta_{i_\alpha, i_\beta, j_\alpha, j_\beta} \delta_{i_{\alpha+1},j_{\beta+1}} \cdots
\delta_{i_r,j_r} \big)_{(\bi,\bj) \in I(n,r) \times I(n,r)}.
\]
Here $\delta_{i_\alpha, i_\beta, j_\alpha, j_\beta} = \delta_{i_\alpha,
  i_\beta} \delta_{i_\alpha, j_\alpha} \delta_{i_\alpha, j_\beta}$ is a
generalised Kronecker delta symbol, which is 1 if
$i_\alpha=i_\beta=j_\alpha=j_\beta$ and 0 otherwise. So the matrix
$\Psi_{n,r}(p_{\alpha,\beta})$ is a diagonal matrix with
$(\bi,\bi)$-entry equal to
\[
\begin{cases}
  1 & \text{if } i_\alpha
  = i_\beta \\
  0 & \text{otherwise},
\end{cases}
\]
for $\bi = i_1 \cdots i_r$ in $I(n,r)$.  If we reorder $I(n,r)$ so
that all the nonzero diagonal entries come before the zero ones, then
the matrix $\Psi_{n,r}(p_{\alpha,\beta})$ has the block form
  \[
  \left[\begin{array}{c|c} {\bI} & \bzero \\ \hline \bzero &
    \bzero\end{array}\right]
  \]
and an easy calculation with block matrices shows that the commutant
$\End_{p_{\alpha,\beta}}(\V^{\otimes r})$ of $p_{\alpha,\beta}$
consists of all block matrices of the form
  \[
  \left[\begin{array}{c|c}
    \bA & \bzero \\ \hline \bzero & \bB
  \end{array}\right]
  \]
where $\bA, \bB $ are arbitrary matrices (of the relevant sizes).  So
$p_{\alpha,\beta}$ sends all $\bv_\bi$ satisfying the condition
$i_\alpha = i_\beta$ to a linear combination of $\bv_\bj$ such that
$j_\alpha = j_\beta$ and sends the $\bv_\bi$ satisfying $i_\alpha \ne
i_\beta$ to a linear combination of $\bv_\bj$ such that $j_\alpha \ne
j_\beta$.

It follows that the centraliser algebra $H_\Bbbk(n,r)$, which is the
intersection of the commutants of the various $p_{\alpha,\beta}$ for
$1 \le \alpha < \beta \le r$, is the algebra of all value-type
preserving endomorphisms. This proves part (b).

(c) The proof of part (c) is well known and can be found, for instance, in
\cite{Green:book}.
\end{proof}

\begin{defn}\label{def:slices}
Let $\bi = i_1 \cdots i_r$, $\bj = j_1 \cdots j_r$ be given
multi-indices. Suppose that $\alpha \in \{1, \dots, r\}$ is a
place. The row and column $\alpha$-\emph{slices} of $\bA$ determined
by $(\bi,\bj)$ are respectively the $n$-vectors
  \[
  \big( a^{i_1\cdots i_{\alpha-1} \,i\, i_{\alpha+1}\cdots
    i_r}_{j_1\cdots j_{\alpha-1} \, j_\alpha\, j_{\alpha+1}\cdots
    j_r}\big)_{i = 1, \dots, n} \quad \text{ and } \quad \big(
  a^{i_1\cdots i_{\alpha-1} \,i_\alpha\, i_{\alpha+1}\cdots
    i_r}_{j_1\cdots j_{\alpha-1} \,j\, j_{\alpha+1}\cdots j_r}
  \big)_{j = 1, \dots, n} .
  \]
We denote these vectors respectively by the suggestive shorthand
notations
  \[
  a^{i_1\cdots i_{\alpha-1} \,*\, i_{\alpha+1}\cdots i_r}_{j_1\cdots
    j_{\alpha-1} \, j_\alpha\, j_{\alpha+1}\cdots j_r} \qand a^{i_1\cdots
    i_{\alpha-1} \,i_\alpha\, i_{\alpha+1}\cdots i_r}_{j_1\cdots
    j_{\alpha-1} \,*\, j_{\alpha+1}\cdots j_r} \, .
  \]
Putting a $\sum$ symbol in front of a slice implies a sum over
the elements of the slice.  We also extend this notation in the
obvious way to allow two or more $*$s to appear; we call them
\emph{double slices}, etc.
\end{defn}

Here are some basic properties of invariants in $\E_\Bbbk(n,r)$, for
$r \ge 2$. (If $r=1$, invariants in $\E_\Bbbk(n,r)$ are just $n
\times n$ GDS matrices; see Section \ref{sec:GDS}.)

\begin{prop}\label{prop:inv-properties}
  Suppose that $r \ge 2$ and that $\bA$ is an element of
  $\E_\Bbbk(n,r)$. Then:
  \begin{enumerate}
  \item For any $\bp$, $\bq \in I(n,r-1)$, the scalars
    $b^\bp_\bq(\alpha)$ appearing in Proposition \ref{prop:GHS}(a) are
    independent of $\alpha$. That is, all the slice sums determined by
    $\bp, \bq$ have the same value $b^\bp_\bq$.
  \item For any $\bp$, $\bq \in I(n,r-1)$, let $\bA^\bp_\bq$ be the $n
    \times n$ block $\bA^\bp_\bq = (a^{\bp\, i}_{\bq\, j})_{i,j = 1,
      \dots, n}$. Then $\bA^\bp_\bq$ is GDS, with common row and
    column sums equal to $b^\bp_\bq$.
  \item If $\bi = i_1 \cdots i_r$, $\bj = j_1 \cdots j_r$ are in
    $I(n,r)$ and $\vt(\bi) \ne \vt(\bj)$ then $a^\bi_\bj = 0$.
  \item If $\bi = i_1 \cdots i_r$, $\bj = j_1 \cdots j_r$ are in
    $I(n,r)$, $\vt(\bi) = \vt(\bj)$, and $i_r$ appears in $\bp = i_1
    \cdots i_{r-1}$, then $j_r$ also appears in the same place in $\bq
    = j_1 \cdots j_{r-1}$, and $a^\bi_\bj = b^\bp_\bq$.
  \end{enumerate}
\end{prop}

\begin{proof}
(a) Since $\E_\Bbbk(n,r)$ is the intersection of $G_\Bbbk(n,r)$,
$H_\Bbbk(n,r)$, and $S_\Bbbk(n,r)$ it is clear that $\bA$ must be
constant on place-permutation orbits, by Proposition
\ref{prop:GHS}(c). This immediately implies that its slice sums are
independent of $\alpha$.

(b) This follows from part (a). Explicitly, for any $\alpha = 1,
\dots, r$ and any given $\bp = i_1 \cdots i_{\alpha-1} i_{\alpha+1}
\cdots i_r$, $\bq = j_1 \cdots j_{\alpha-1} \, j_{\alpha+1} \cdots
j_r$ in $I(n,r-1)$, part (a) says that there exists a scalar
$b^\bp_\bq$ in $\Bbbk$ such that
\[
\sum a^{i_1\cdots i_{\alpha-1} \,*\, i_{\alpha+1}\cdots
  i_r}_{j_1\cdots j_{\alpha-1} \, j \, j_{\alpha+1}\cdots j_r} =
b^\bp_\bq \qand \sum a^{i_1\cdots i_{\alpha-1} \, i \,
  i_{\alpha+1}\cdots i_r}_{j_1\cdots j_{\alpha-1} \,*\,
  j_{\alpha+1}\cdots j_r} = b^\bp_\bq
\]
for any $i,j = 1, \dots, n$.  Part (b) is the particular case where
$\bp = i_1 \cdots i_{r-1}$, $\bq = j_1 \cdots j_{r-1}$.

(c) This is the same as part (b) of Proposition \ref{prop:GHS}, repeated
here for the sake of convenience.

(d) This follows from parts (b) and (c). By part (c), under these
hypotheses, all terms except $a^{\bp\,i_r}_{\bq\,j_r}$ of the slice
sum $\sum a^{\bp\,i_r}_{\bq\,*} = b^{\bp}_{\bq}$ must be zero.
\end{proof}

\begin{defn}\label{def:restr}
  Assume that $r \ge 1$. Let $\bA \in \E_\Bbbk(n,r)$ be given.  For
  any $\bp$, $\bq \in I(n,r-1)$ let $b^\bp_\bq$ be the common value of
  the slice sums in $\bA$ indexed by $\bp, \bq$. (This is a single
  scalar if $r=1$.)  The matrix $\bB = (b^\bp_\bq)_{\bp,\bq \in
    I(n,r-1)}$ in $\Mat_{I(n,r-1)}(\Bbbk)$ is the \emph{restriction}
  of $\bA$. We write $\rho(\bA) = \bB$ for the restriction.
\end{defn}

Proposition \ref{prop:inv-properties}(b) says that the invariant
$\bA$ is obtained by ``blowing up'' its restriction $\bB = \rho(\bA)$
by a process which replaces each matrix entry $b^\bp_\bq$ of $\bB$ by
an $n \times n$ GDS matrix with row and column sums equal to
$b^\bp_\bq$. Of course, $\bA$ must also be invariant under
place-permutations, so blowing up is not the only consideration.

\begin{prop}\label{prop:restr-is-invariant}
  Let $r \ge 1$, and suppose that $\bA \in \E_\Bbbk(n,r)$. Then the
  restriction $\rho(\bA)$ belongs to $\E_\Bbbk(n,r-1)$. Thus, $\rho$
  is a $\Bbbk$-linear map from $\E_\Bbbk(n,r)$ into $\E_\Bbbk(n,r-1)$.
\end{prop}

\begin{proof}
Set $\bB = \rho(\bA)$. If $r=1$ then $\bB$ is just a scalar and there
is nothing to prove, since $\E_\Bbbk(n,0) = \Bbbk$. For the moment, we
assume that $r=2$. Given any $i_1i_2 \in I(n,2)$, we compute the
double-slice sum $\sum a^{i_1i_2}_{**}$ two ways, by applying the
independence in Proposition \ref{prop:inv-properties} and changing
the order of summation:
\begin{align*}
\sum a^{i_1i_2}_{**} &= \sum_{j=1}^n \sum_{j'=1}^n a^{i_1i_2}_{jj'} =
\sum_{j=1}^n b^{i_1}_j = \sum b^{i_1}_*\\
&= \sum_{j'=1}^n \sum_{j=1}^n a^{i_1i_2}_{jj'} =
\sum_{j=1}^n b^{i_2}_{j'} = \sum b^{i_2}_* .
\end{align*}
Thus all the row sums in $\bB$ have a common value. Similarly, by
considering the double-slice sum $\sum a^{**}_{j_1j_2}$ we see that
$\sum b^*_{j_1} = \sum b^*_{j_2}$, for arbitrary $j_1j_2$ in $I(n,2)$,
so all the column sums in $\bB$ have a common value. Finally, for
arbitrary $i_1, j_2 = 1, \dots, n$ we compute the mixed-double-slice
sum $\sum a^{i_1*}_{*j_2}$ two ways:
\begin{align*}
\sum a^{i_1*}_{*j_2} &= \sum_{i=1}^n \sum_{j=1}^n a^{i_1i}_{jj_2} =
\sum_{i=1}^n b^{i}_{j_2} = \sum b^{*}_{j_2}\\
&= \sum_{j=1}^n \sum_{i=1}^n a^{i_1i}_{jj_2} =
\sum_{j=1}^n b^{i_1}_{j} = \sum b^{i_1}_* .
\end{align*}
Thus $\sum b^{*}_{j_2} = \sum b^{i_1}_{*}$. This shows that all the
row sums are equal to all the column sums in $\bB$. In other words,
$\bB$ is GDS, and thus belongs to $\E_\Bbbk(n,1)$. This proves the
result in case $r=2$.

Now assume that $r>2$. We need to show that $\bB$ satisfies the
conditions of Proposition \ref{prop:GHS}(a)--(c). It is easy to see
that conditions (b), (c) hold for $\bB$ since they hold for
$\bA$. Then calculations similar to the above (and place-permutation
invariance) show that $\bB$ satisfies the slice-sum equations in part
(a). Thus $\bB$ belongs to $\E_\Bbbk(n,r-1)$ as claimed.
\end{proof}

Going forward, we will often regard a given $\bA = (a^{i_1
  \cdots i_r}_{j_1 \cdots j_r})_{i_1 \cdots i_r, j_1 \cdots j_r \in
  I(n,r)}$ in $\E_\Bbbk(n,r)$ as the $n \times n$ block matrix
\begin{equation}\label{eq:blocks}
\bA = (\bA^i_j)_{i,j = 1, \dots,n} = 
\begin{bmatrix}
  \bA^1_1 & \cdots & \bA^1_n \\
  \vdots & \ddots & \vdots \\
  \bA^n_1 & \cdots & \bA^n_n
\end{bmatrix}
\end{equation}
where each block $\bA^i_j$ is defined by $\bA^i_j = (a^{i\, i_2 \cdots
  i_r}_{j\, j_2 \cdots j_r})_{i_2 \cdots i_r, j_2 \cdots j_r \in
  I(n,r-1)}$.

\begin{rmk}\label{rmk:block-sums}
The block notation just introduced provides a convenient description
of the restriction map. Given $\bA$ in $\E_\Bbbk(n,r)$, we have
\[
\rho(\bA) = \textstyle \sum \bA^i_* = \sum \bA^*_j
\]
for any $i,j$.  That is, the restriction $\bB=\rho(\bA)$ is the common
value of the block row and column sums in the block matrix
\eqref{eq:blocks}.
\end{rmk}

\begin{example}\label{ex:n23-r2-inv}
We now illustrate how, using only the value-type condition of
Proposition \ref{prop:GHS}(b) we can obtain general forms for the
invariants $\E_\Bbbk(n,r)$ for small values of $n$ and $r$ (we will
use blank entries to denote entries which are zero due to value-type
mismatches).  The general form of invariants in $\E_\Bbbk(2,2)$,
$\E_\Bbbk(3,2)$ are displayed below, respectively:
\[
\setlength{\arraycolsep}{2.7pt}\scriptsize
\begin{array}{{c|cc|cc|}}
  &\rt{11}&\rt{12}&\rt{21}&\rt{22}\\ \hline
  11&*& & &*\\
  12& &*&*& \\ \hline
  21& &*&*& \\
  22&*& & &*\\ \hline
\end{array} \; ,
\qquad
\begin{array}{c|ccc|ccc|ccc|}
  &\rt{11}&\rt{12}&\rt{13}&\rt{21}&\rt{22}&\rt{23}&
  \rt{31}&\rt{32}&\rt{33} \\ \hline
  11&*& &  & &*&  & & &*\\
  12& &*&* &*& &* &*&*& \\
  13& &*&* &*& &* &*&*& \\ \hline
  21& &*&* &*& &* &*&*& \\  
  22&*& &  & &*&  & & &*\\
  23& &*&* &*& &* &*&*& \\ \hline
  31& &*&* &*& &* &*&*& \\
  32& &*&* &*& &* &*&*& \\
  33&*& &  & &*&  & & &*\\  \hline
\end{array}\; .
\]
  The general
form of invariants in $\E_\Bbbk(4,2)$ is displayed below:
\[
\setlength{\arraycolsep}{2.7pt}\scriptsize
\begin{array}{c|cccc|cccc|cccc|cccc|}
  &\rt{11}&\rt{12}&\rt{13}&\rt{14}&\rt{21}&\rt{22}&\rt{23}&\rt{24}&
  \rt{31}&\rt{32}&\rt{33}&\rt{34}&\rt{41}&\rt{42}&\rt{43}&\rt{44}\\ \hline
  11&*& & & & &*& & & & &*& & & & &*\\
  12& &*&*&*&*& &*&*&*&*& &*&*&*&* & \\
  13& &*&*&*&*& &*&*&*&*& &*&*&*&*& \\
  14& &*&*&*&*& &*&*&*&*& &*&*&*&*& \\  \hline
  21& &*&*&*&*& &*&*&*&*& &*&*&*&*&  \\  
  22&*& & & & &*& & & & &*& & & & &*\\
  23& &*&*&*&*& &*&*&*&*& &*&*&*&*& \\  
  24& &*&*&*&*& &*&*&*&*& &*&*&*&*& \\ \hline  
  31& &*&*&*&*& &*&*&*&*& &*&*&*&*& \\
  32& &*&*&*&*& &*&*&*&*& &*&*&*&*& \\
  33&*& & & & &*& & & & &*& & & & &*\\  
  34& &*&*&*&*& &*&*&*&*& &*&*&*&*& \\ \hline
  41& &* &* &* &*& &*&*&*&*& &*&*&*&*& \\
  42& &* &* &* &*& &*&*&*&*& &*&*&*&*& \\
  43& &* &* &* &*& &*&*&*&*& &*&*&*&*& \\
  44&* & & & & &*& & & & &*& & & & &*\\ \hline
\end{array} \;.
\]
In these depictions, starred entries can be non-zero, but they are not
arbitrary, because they must be invariant under place-permutations and
their slice sums must satisfy the GDS conditions in Proposition
\ref{prop:inv-properties}(b).
\end{example}

Let $\bP(w) = \Phi_{n,1}(w)$ be the permutation matrix representing $w
\in W_n$. As a linear endomorphism of $\V$, $\bP(w)$ is the linear map
sending $\bv_j$ to $\bv_{w(j)}$, for $j = 1, \dots, n$. In terms of
matrix coordinates, $\bP(w) = ( \delta_{i, w(j)} )_{i,j = 1,
  \dots, n}$. It follows that
\begin{equation}\label{eq:Kron-power}
\Phi_{n,r}(w) = \bP(w)^{\otimes r} = \big( \delta_{i_1, w(j_1)} \cdots
\delta_{i_r, w(j_r)} \big)_{i_1 \cdots i_r}^{j_1 \cdots j_r}.
\end{equation}
It is easy to compute the restriction of such matrices.

\begin{lem}\label{lem:rho-of-Kron}
  For any $w \in W_n$ and any $r \ge 1$, $\rho(\bP(w)^{\otimes r}) =
  \bP(w)^{\otimes(r-1)}$. In particular, $\rho(\bP(w)) = 1$.
\end{lem}

\begin{proof}
Write $\bA = \bP(w)^{\otimes r}$.  Express $\bA$ as an $n \times n$
block matrix $\bA = (\bA^i_j)$ as in \eqref{eq:blocks}. Then by
\eqref{eq:Kron-power}, $\bA^i_j$ is given by
\[
\bA^i_j = \delta_{i, w(j)} \, \bP(w)^{\otimes (r-1)}, \quad \text{all
} i,j = 1, \dots, n.
\]
The result now follows by Remark \ref{rmk:block-sums}. 
\end{proof}

Since the actions of $W_n$ and $\Ptn_r(n)$ on $\V^{\otimes r}$
commute, we have induced left and right actions of $W_n$ on
$\E_\Bbbk(n,r) = \End_{\Ptn_r(n)}(\V^{\otimes r})$, which are given by
left and right multiplication of the corresponding permutation
matrices. Explicitly,
\begin{equation}\label{eq:W_n-actions}
  (w, \bA) \mapsto \Phi(w) \bA = \bP(w)^{\otimes r} \bA, \quad
  (\bA,w) \mapsto \bA \Phi(w)  = \bA \bP(w)^{\otimes r} 
\end{equation}
defines the left and right actions, respectively. In other words, the
algebra $\E_\Bbbk(n,r)$ is stable under row and column permutations by
$r$th Kronecker powers of $\bP(w)$ for $w \in W_n$.

\section{Description of the invariants $\E_\Bbbk(n,r+\fhalf)$}%
\label{sec:special}\noindent
We now study how the algebra $\E_\Bbbk(n,r+\fhalf)
= \End_{\Ptn_{r+\half}(n)}(\V^{\otimes r})$ is related to
$\E_\Bbbk(n,r) = \End_{\Ptn_r(n)}(\V^{\otimes r})$. Recall that in
this context we identify $\V^{\otimes r}$ with $\V^{\otimes r}
\otimes \bv_n$.  The following terminology is useful for our study.

\begin{defn}\label{def:special}
(i) If $i_1 \cdots i_r \in I(n,r)$, and if $j \in \{1, \dots, n\}$,
let $\Lambda_j(i_1 \cdots i_r)$ be the set of places in which the
value $j$ appears:
\[
\Lambda_j(i_1 \cdots i_r) = \{ \alpha \in \{1, \dots, r\} : i_\alpha =
j \}.
\]
We say that $i_1\cdots i_r$ \emph{contains} $j$ if $\Lambda_j(i_1
\cdots i_r)$ is non-empty.

(ii) Fix $i,j$ such that $1 \le i,j \le n$. Let $\E_\Bbbk(n,r)^i_j$ be
the set of invariants $\bA = (a^{i_1 \cdots i_r}_{j_1 \cdots j_r})$ in
$\E_\Bbbk(n,r)$ satisfying the following condition:
\[
\text{if $a^{i_1 \cdots i_r}_{j_1 \cdots j_r} \ne 0$ then
  $\Lambda_i(i_1 \cdots i_r) = \Lambda_j(j_1 \cdots j_r)$.}
\]
We call elements of any $\E_\Bbbk(n,r)^i_j$ \emph{special
  invariants}. Note that $\E_\Bbbk(n,r)^i_j$ is a $\Bbbk$-module, for
any $i,j$. For any $i = 1, \dots, n$, $\E_\Bbbk(n,r)^i_i$ is an
algebra over $\Bbbk$, the algebra of $\bv_i$-fixing invariants.
\end{defn}

\begin{lem}\label{lem:special-criterion}
  Suppose that $r\ge 1$ and $\bA$ is in $\E_\Bbbk(n,r)$. If all blocks
  except for $\bA^i_j$ in the $ith$ block row and $j$th block column
  are zero blocks, then $\bA$ must be a special invariant in
  $\E_\Bbbk(n,r)^i_j$.
\end{lem}

\begin{proof}
Suppose $a^{i_1 \cdots i_r}_{j_1\cdots j_r} \ne 0$. We must show
that $\Lambda_i(i_1 \cdots i_r) = \Lambda_j(j_1 \cdots j_r)$.

\emph{Case 1.} Assume that $i_1 \cdots i_r$ contains $i$ or $j_1
\cdots j_r$ contains $j$. Since $\bA$ is invariant under
place-permutations we can assume that $i_1=i$ or $j_1=j$. Then the
hypothesis implies that \emph{both} $i_1=i$ and $j_1=j$. Since $\bA$
preserves value-type, it follows that the places in $i_1 \cdots i_r$
containing $i$ must agree with the places in $j_1 \cdots j_r$
containing $j$, so $\Lambda_i(i_1 \cdots i_r) = \Lambda_j(j_1 \cdots
j_r)$.

\emph{Case 2.} Otherwise, $i_1 \cdots i_r$ does not contain $i$ and
$j_1 \cdots j_r$ does not contain $j$. Again $\Lambda_i(i_1 \cdots
i_r) = \Lambda_j(j_1 \cdots j_r)$, as both sets are empty.
\end{proof}

The next result says in particular that the blocks of any invariant
are always special invariants in the previous degree.

\begin{lem}\label{lem:blocks-are-special}
  Suppose that $r \ge 1$ and that $\bA$ is in $\E_\Bbbk(n,r)$. For any
  fixed $1 \le i,j \le n$, the block matrix $\bA^i_j = (a^{i\, i_2
    \cdots i_r}_{j\,j_2 \cdots j_r})_{i_2 \cdots i_r,\, j_2 \cdots j_r
    \in I(n,r-1)}$ belongs to $\E_\Bbbk(n,r-1)^i_j$, after re-indexing
  its rows and columns via the forgetful map that omits the initial
  term of each multi-index.
\end{lem}

\begin{proof}
Clearly $\bA^i_j$ belongs to $\E_\Bbbk(n,r-1)$, since it satisfies the
conditions of Proposition \ref{prop:GHS}. Furthermore, the fact that
$\bA$ preserves value-type implies that
\[
\text{ if } a^{i\, i_2 \cdots i_r}_{j\,j_2 \cdots j_r} \ne 0 \text{
  then } \Lambda_i(i_2 \cdots i_r) = \Lambda_j(j_2 \cdots j_r)
\]
since $\Lambda_i(i\, i_2 \cdots i_r) = \Lambda_j(j\,j_2 \cdots j_r)$
by hypothesis.  So $\bA^i_j \in \E_\Bbbk(n,r-1)^i_j$, as
required.
\end{proof}

\begin{lem}\label{lem:half-iso}
$\E_\Bbbk(n,r+\fhalf)$ is isomorphic to $\E_\Bbbk(n,r)^n_n$ (as
  algebras). In particular, $\E_\Bbbk(n,r+\fhalf)$ embeds in
  $\E_\Bbbk(n,r)$.
\end{lem}

\begin{proof}
This is more or less immediate from the definitions. Thanks to the
identification of $\V^{\otimes r}$ with $\V^{\otimes r} \otimes
\bv_n$, an invariant in
\[
\E_\Bbbk(n,r+\fhalf) = \End_{\Ptn_{r+\,1/2}(n)}(\V^{\otimes r})
\]
is a $\Bbbk$-linear endomorphism of $\V^{\otimes r} \otimes \bv_n$
commuting with the action of $\Ptn_{r+\,1/2}(n)$. So it must preserve
value-type, which means that it must fix $\bv_n$ in all tensor places,
since it does so in the last place. Also, it is constant on
place-permutation orbits for $\Sym_r$ acting on the first $r$ places,
and satisfies the slice equations in Proposition \ref{prop:GHS}(a) in
all places. If we index rows and columns of the invariant by elements
of $I(n,r)$, by forgetting the last tensor factor (of $\bv_n$), then
we get an invariant in $\E_\Bbbk(n,r)$.
\end{proof}

\begin{example}\label{ex:special-inv}
A special invariant in $\E_\Bbbk(4,2)^4_4 \cong \E_\Bbbk(4,2+\fhalf)$
is of the form
\[
\setlength{\arraycolsep}{2.7pt}\scriptsize
\begin{array}{c|cccc|cccc|cccc|cccc|}
  &\rt{11}&\rt{12}&\rt{13}&\rt{14}&\rt{21}&\rt{22}&\rt{23}&\rt{24}&
  \rt{31}&\rt{32}&\rt{33}&\rt{34}&\rt{41}&\rt{42}&\rt{43}&\rt{44}\\ \hline
  11&*& & & & &*& & & & &*& & & & & \\ 12& &*&*& &*& &*& &*&*& & & & &
  & \\ 13& &*&*& &*& &*& &*&*& & & & & & \\ 14& & & &*& & & &*& & &
  &*& & & & \\ \hline 21& &*&*& &*& &*& &*&*& & & & & & \\ 22&*& & & &
  &*& & & & &*& & & & & \\ 23& &*&*& &*& &*& &*&*& & & & & & \\ 24& &
  & &*& & & &*& & & &*& & & & \\ \hline 31& &*&*& &*& &*& &*&*& & & &
  & & \\ 32& &*&*& &*& &*& &*&*& & & & & & \\ 33&*& & & & &*& & & &
  &*& & & & & \\ 34& & & &*& & & &*& & & &*& & & & \\ \hline 41& & & &
  & & & & & & & & &*&*&*& \\ 42& & & & & & & & & & & & &*&*&*& \\ 43&
  & & & & & & & & & & & &*&*&*& \\ 44& & & & & & & & & & & & & & &
  &*\\ \hline
\end{array}
\]
where all blank positions must be zero, and the starred positions can
be non-zero. This should be compared with Example \ref{ex:n23-r2-inv}.
All special invariants in $\E_\Bbbk(4,2)$ look like this, up to a
reordering of rows and columns.

Notice that deleting the rows and columns indexed by labels containing
4 yields the general form of an invariant in $\E_\Bbbk(3,2)$; see
Example \ref{ex:n23-r2-inv}. This observation motivates Proposition
\ref{prop:n-reduce-iso} below.
\end{example}

The reader may wish to refer to Example \ref{ex:special-inv} when
working through the proof of the next result.

\begin{prop}\label{prop:n-reduce-iso}
Suppose that $n \ge 2$. For any $1 \le p,q \le n$, there is a
$\Bbbk$-linear isomorphism
\[
\E_\Bbbk(n,r)^p_q \xrightarrow{\;\approx\;} \E_\Bbbk(n-1,r)
\]
given by respectively excising all rows, columns labeled by a
multi-index containing $p$, $q$ respectively and re-indexing the sets
$\{1, \dots, p-1, p+1, \dots, n\}$ and $\{1, \dots, q-1, q+1 \dots,
n\}$ to match $\{1, \dots, n-1\}$. In particular, we have a
$\Bbbk$-linear isomorphism
\[
\E_\Bbbk(n,r+\fhalf) \xrightarrow{\;\approx\;} \E_\Bbbk(n-1,r)
\]
given by excising all rows and columns labeled by a multi-index
containing $n$ (with no re-indexing needed).
\end{prop}

\begin{proof}
We first prove the special case in which $p=q=n$. Suppose that $\bA
\in \E_\Bbbk(n,r)^n_n$. We obtain a corresponding invariant
$\eta(\bA) \in \E_\Bbbk(n-1,r)$ by excising all rows and columns of
$\bA$ indexed by a label containing $n$. This defines a $\Bbbk$-linear
map $\eta: \E_\Bbbk(n,r)^n_n \to \E_\Bbbk(n-1,r)$.

For the opposite direction, suppose that $\bC \in \E_\Bbbk(n-1,r)$ is
given. We define a $\Bbbk$-linear map $\theta_r: \E_\Bbbk(n-1,r) \to
\E_\Bbbk(n,r)^n_n$ by induction on $r$, holding $n$ fixed. If $r=1$,
we set $\theta_1(\bC) = \bA = (a^i_j)_{i,j = 1, \dots, n}$, where
\[
a^i_j = 
\begin{cases}
  c^i_j & \text{ if $i \ne n$ and $j \ne n$} \\
  s & \text{ if $i = n$ and $j = n$} \\
  0 & \text{ otherwise.}
\end{cases}
\]
Here, $s$ is the common value of the row and column sums in the given
GDS matrix $\bC$.  If $r>1$, we regard $\bC = (\bC^i_j)_{i,j = 1,
  \dots, n-1}$ as an $(n-1) \times (n-1)$ block matrix, where each
$\bC^i_j = (b^{i\, i_2 \cdots i_r}_{j\, j_2 \cdots j_r})_{i_2 \cdots
  i_r, j_2 \cdots j_r \in I(n-1,r)}$, and then we set $\theta_r(\bC) =
\bA = (\bA^i_j)_{i,j = 1, \dots, n}$, again as a block matrix, where
the blocks $\bA^i_j = (a^{i\, i_2 \cdots i_r}_{j\,j_2 \cdots j_r})_{i_2
  \cdots i_r, j_2 \cdots j_r \in I(n,r-1)}$ are given by
\[
\bA^i_j =
\begin{cases}
  \theta_{r-1}(\bC^i_j) & \text{ if $i \ne n$ and $j \ne n$} \\
  \mathbf{S} & \text{ if $i=n$ and $j=n$} \\
  \bzero & \text{ otherwise.}
\end{cases}
\]
Here, $\mathbf{S}$ is the common value of the block sum of the first
$n-1$ block rows and columns in $\bA$; that is, $\mathbf{S} = \sum
\bA^*_j = \sum \bA^i_*$, for any $i, j = 1, \dots,
n-1$. Alternatively, $\mathbf{S} = \theta_{r-1}(\mathbf{S}')$, where
$\mathbf{S}' = \sum \bC^i_* = \sum \bC^*_j$ for any $i,j = 1, \dots,
n-1$.

Having defined $\theta = \theta_r$, we claim that $\theta$ is a two-sided
inverse of $\eta$, so $\eta$ is the desired isomorphism (and
$\eta^{-1} = \theta$). Details are left to the reader.  This proves
the special case.  The general case, for arbitrary $1 \le p,q \le n$,
follows from the special case by re-indexing (interchange $n$ with $p$
and $q$ in row and column indices, respectively). 
\end{proof}

\begin{rmk}\label{rmk:eta-theta}
For general $1 \le p,q \le n$, whenever necessary we
will denote the $\Bbbk$-linear isomorphisms $\eta$, $\theta$ in the
above proof by $\eta^p_q$, $\theta^p_q$ respectively. 
\end{rmk}

Lemma \ref{lem:blocks-are-special} tells us that blocks of any
invariant are always special invariants. If the given invariant is
itself special, then we can be more precise about the nature of its
blocks. Let $\ov{p}$ be the image of $p$ under the renumbering
bijection $\{1, \dots, i-1, i+1, \dots, n\} \cong \{1, \dots, n-1\}$;
similarly $\ov{q}$ is the image of $q$ under $\{1, \dots, j-1, j+1,
\dots, n\} \cong \{1, \dots, n-1\}$.

\begin{prop}\label{prop:special-ext-form}
   If $\bA \in \E_\Bbbk(n,r)^i_j$ is a special invariant then with
   $\eta = \eta^i_j$ we have:
   \begin{enumerate}
   \item $\bA^i_q = \bzero$ for any $q \ne j$ and $\bA^p_j = \bzero$ for
     any $p \ne i$.
   \item If $p \ne i$ and $q \ne j$ then $\eta(\bA^p_q) \in
     \E_\Bbbk(n-1,r-1)^{\ov{p}}_{\ov{q}}$\,. 
   \item $\rho(\bA) = \bA^i_j$. Thus $\rho(\bA) \in
     \E_\Bbbk(n,r-1)^i_j$\,.
   \end{enumerate}
\end{prop}

\begin{proof}
 (a) Clear from the definition of $\E_\Bbbk(n,r)^i_j$.

 (b) This follows from Proposition \ref{prop:n-reduce-iso} and Lemma
   \ref{lem:blocks-are-special}, applied to $\eta(\bA)$.

 (c) Set $\bB = \rho(\bA)$. By Remark \ref{rmk:block-sums}, all the
   block row and column sums of $\bA$ are equal to $\bB$. Thus, by part
   (a), we have $\bB = \bA^i_j$. The last statement in (c) then follows
   by Lemma \ref{lem:blocks-are-special}.
\end{proof}

Note that Proposition \ref{prop:special-ext-form} is well illustrated
by Example \ref{ex:special-inv}.

We finish this section with the following observation. 

\begin{lem}\label{lem:theta-rho-commute}
  With $\theta = \theta^i_j$ and $\eta = \eta^i_j$ the
  following diagram commutes:
   \[
   \begin{tikzcd}[column sep = large]
   \E_\Bbbk(n-1,r) \arrow{r}{\theta} \arrow{d}{\rho}
   & \E_\Bbbk(n,r)^i_j \arrow{d}{\rho}\\
   \E_\Bbbk(n-1,r-1) \arrow{r}{\theta}& \E_\Bbbk(n,r-1)^i_j 
   \end{tikzcd} \; .
   \]
   In other words, the restriction map $\rho$ commutes with $\theta$.
   (It also commutes with $\eta = \theta^{-1}$.)
\end{lem}

\begin{proof}
Left to the reader.
\end{proof}

Restriction $\rho: \E_\Bbbk(n,r) \to \E_\Bbbk(n,r-1)$ gives a way of
obtaining invariants in degree $r-1$ from invariants in degree $r$
(for $r\ge 1$). The opposite problem is the \emph{extension problem:}
\begin{quote}
  Given $\bB$ in $\E_\Bbbk(n,r-1)$, find some $\bA$ in $\E_\Bbbk(n,r)$
  such that $\rho(\bA) = \bB$.
\end{quote}
A closely related problem is the \emph{decomposition problem:}
\begin{quote}
  Given $\bA$ in $\E_\Bbbk(n,r)$, write $\bA$ as a sum of special
  invariants.
\end{quote}
We will prove in Theorem \ref{thm:properties} ahead that both problems
can always be solved. We show now that this implies the main result of
this paper.

\begin{thm}\label{thm:main1}
  Let $\Bbbk$ be a commutative ring. For any $n \ge 2$, $r \ge
  1$ the maps $\Phi_{n,r}: \Bbbk W_n \to \E_\Bbbk(n,r)$ and
  $\Phi_{n,r+\half}: \Bbbk W_{n-1} \to \E_\Bbbk(n,r+\fhalf)$ are
  surjective.
\end{thm}

\begin{proof}
By Proposition \ref{prop:n-reduce-iso}, the surjectivity of
$\Phi_{n,r+\half}$ follows from the surjectivity of $\Phi_{n-1,r}$, so
it suffices to prove the surjectivity of $\Phi_{n,r}$. This
surjectivity is trivial if $r=0$ since $\E_\Bbbk(n,r) \cong \Bbbk$. We
proceed by induction on $n$.  Let $\bA \in \E_\Bbbk(n,r)$. By Theorem
\ref{thm:properties}, we can write
\[
\bA = \bA(1) + \cdots + \bA(n)
\]
where $\bA(j)$ is in $\E_\Bbbk(n,r)^n_j$ for each $j = 1, \dots,
n$. By Proposition \ref{prop:n-reduce-iso} and the inductive
hypothesis, each $\eta^n_j \bA(j)$ belongs to the image of
$\Phi_{n-1,r}$, for $j = 1, \dots, n$. This implies that each $\bA(j)
= \theta^n_j \eta^n_j \bA(j)$ belongs to the image of
$\Phi_{n,r}$. Hence so does $\bA$, and the proof is complete.
\end{proof}

\begin{rmk}\label{rmk:alt-pf}
Another proof of Theorem \ref{thm:main1} is based on the existence of
extensions (also proved in Theorem \ref{thm:properties}). First note
that it is easy to prove Theorem \ref{thm:main1} if $n=r$, because
after all we just need to solve the equation
\[
\bA = \textstyle \sum_{w \in W_n} x_w \Phi_{n,r}(w). 
\]
When $n=r$ the equation has at most one solution, given by setting
\[
x_w = a_{12\cdots n}^{w(1)w(2)\cdots w(n)} \quad \text{for each $w \in W_n$}.
\]
This works because only one permutation $w$ can contribute to that
entry in the matrix $\bA$.  It is not difficult to check that this is
actually a solution.  It follows immediately (given the existence of
extensions) that the ``same'' linear combination
\[
\bA' = \sum_{w \in W_n} x_w \Phi_{n,r+1}(w)
\]
is the unique extension of $\bA$ in $\E_\Bbbk(n,n+1)$. Repeating the
argument inductively, we see that Theorem \ref{thm:main1} holds for
all $r \ge n$.

Assume now that $r<n$, which is the difficult case. The existence of
extensions implies in particular that the restriction map $\rho$ is
always surjective.  Thus the commutative diagram
\[
\begin{tikzcd}
\E_\Bbbk(n,n) \arrow[r, "\rho^{n-r}"] & \E_\Bbbk(n,r) \\
\Bbbk W_n \arrow[u, "\Phi_{n,n}"] \arrow[ur, "\Phi_{n,r}"']
\end{tikzcd}
\]
expresses the map $\Phi_{n,r}$ as the composite of two surjections,
hence it is a surjection, and the theorem is proved. The idea behind
this proof is: given $\bA$ in $\E_\Bbbk(n,r)$ we first extend up to
degree $n$, where we can read off a solution, and then restrict it
back down to degree $r$.
\end{rmk}

\begin{cor}\label{cor:ker-of-rho}
  The kernel of restriction $\rho: \E_\Bbbk(n,r) \to \E_\Bbbk(n,r-1)$
  is isomorphic to $\Phi_{n,r}(\ker \Phi_{n,r-1}) \cong (\ker
  \Phi_{n,r-1})/(\ker \Phi_{n,r})$.
\end{cor}

\begin{proof}
Since $\Phi_{n,r}$ is surjective (by Theorem \ref{thm:main1}) and
$\rho$ is surjective (Remark \ref{rmk:alt-pf}), the commutativity of
the diagram
\[
\begin{tikzcd}
\E_\Bbbk(n,r) \arrow[r, "\rho"] & \E_\Bbbk(n,r-1) \\
\Bbbk W_n \arrow[u, "\Phi_{n,r}"] \arrow[ur, "\Phi_{n,r-1}"']
\end{tikzcd}
\]
implies that $\ker \rho \cong \Phi_{n,r}(\ker \Phi_{n,r-1})$. 
\end{proof}

In \cite{BDM:second}, the authors find an explicit cellular basis for
the kernel of $\Phi_{n,r}$ for all $n,r$. This means that we have an
explicit basis for the kernel of $\rho$ in Corollary
\ref{cor:ker-of-rho}.

\section{Extensions and decompositions}\label{sec:extensions}\noindent
It remains to prove the existence of extensions and
decompositions. This is the purpose of Sections \ref{sec:extensions}
and \ref{sec:proof}. 

First we discuss the extension problem. Let $\bB$ be a given fixed
invariant in $\E_\Bbbk(n,r-1)$. Suppose that $\bA \in \E_\Bbbk(n,r)$
is an extension of the given $\bB$, and write $\bA =
(a^\bi_\bj)_{\bi,\bj \in I(n,r)}$. Then the matrix coordinates
$a^\bi_\bj$ of $\bA$ must satisfy the conditions:
\begin{align}
  & a^{(i_1\cdots i_{r-1} i_{r-1})^\sigma}_{(j_1\cdots j_{r-1}
    j_{r-1})^\sigma} = b^{i_1\cdots i_{r-1}}_{j_1\cdots j_{r-1}} \text{ for
      all } \sigma \in \Sym_r \, , \tag{I-1} \\
  & a^{i_1\cdots i_r}_{j_1\cdots j_r} = 0 \text{ whenever }
  \vt(i_1\cdots i_r) \ne \vt(j_1 \cdots j_r) \hspace{2cm}\tag{I-2}
\end{align}
for all $i_1\cdots i_{r-1}$, $j_1\cdots j_{r-1}$ in $I(n,r-1)$ and all
$i_1\cdots i_r$, $j_1\cdots j_r$ in $I(n,r)$.

Condition (I-1) comes from Proposition \ref{prop:inv-properties}(d),
while condition (I-2) restates value-type preservation, from
Proposition \ref{prop:GHS}(b) (also Proposition
\ref{prop:inv-properties}(c)).  We call (I-1), (I-2)
\emph{initialisation conditions}. They determine the value of all
entries $a^\bi_\bj = a^{i_1 \cdots i_r}_{j_1\cdots j_r}$ for which
either $\#(\bi)<r$ or $\#(\bj)<r$, where we define $\#(\bi)$ to be the
number of distinct values appearing in the multi-index $\bi$. 

Thus, to find an extension $\bA$ of the given $\bB$, we start by
initialising $\bA \in \Mat_{I(n,r)}(\Bbbk)$ to satisfy (I-1),
(I-2). Then it only remains to assign values to the $a^{i_1 \cdots
  i_r}_{j_1 \cdots j_r}$ for which $\#(i_1 \cdots i_r) = r = \#(j_1
\cdots j_r)$.  In other words, if we set
\[
I'(n,r) = \{\bi \in I(n,r): \#(\bi) = r\}
\]
then we can focus just on how to assign entries $a^\bi_\bj$ such that
$\bi,\bj \in I'(n,r)$. By Proposition \ref{prop:GHS}, those entries of
$\bA$ must satisfy the following conditions.  For any $\bp = i_1
\cdots i_{\alpha-1} i_{\alpha+1} \cdots i_r$, $\bq = j_1 \cdots
j_{\alpha-1} j_{\alpha+1} \cdots j_r$ in $I'(n,r-1)$,
\begin{align}
\textstyle \sum a^{i_1\cdots i_{\alpha-1} \, i
  \, i_{\alpha+1}\cdots i_r}_{j_1\cdots j_{\alpha-1} \,*\,
  j_{\alpha+1}\cdots j_r} = b^\bp_\bq \text{ for all } i = 1, \dots,
n\label{eq:slice-row} \\
\textstyle \sum a^{i_1\cdots i_{\alpha-1} \,*\, i_{\alpha+1}\cdots
  i_r}_{j_1\cdots j_{\alpha-1} \, j \, j_{\alpha+1}\cdots j_r} =
b^\bp_\bq \text{ for all } j = 1, \dots,
n \label{eq:slice-col}\\
\intertext{and for all $\bi = i_1 \cdots i_r$,
  $\bj = j_1 \cdots j_r$ in $I'(n,r)$,}
a^{i_1 \cdots i_r}_{j_1 \cdots j_r} - a^{(i_1 \cdots
  i_r)^\sigma}_{(j_1 \cdots j_r)^\sigma} = 0 \text{ all } \sigma \in
\Sym_r \label{eq:stability}.
\end{align}
So finding an extension $\bA$ of the given $\bB$ is (after
initialisation) equivalent to solving the linear system given by
equations \eqref{eq:slice-row}--\eqref{eq:stability}.

Any special invariant $\bA$ in $\E_\Bbbk(n,r)^i_j$ is necessarily an
extension of its block $\bA^i_j$. This follows from Proposition
\ref{prop:special-ext-form}(a) and Remark \ref{rmk:block-sums}. So we
sometimes refer to such special invariants as \emph{special
  extensions}.

Our goal now is to establish the following four interrelated
properties, the first of which is about the existence of
extensions. They will be established by an interleaved double
induction on $n,r$. Note that that each property is based on some
fixed (but arbitrary) block row or column.

\begin{property}[extension property]
  For any $\bB$ in $\E_\Bbbk(n,r-1)$, there exists some $\bA$ in
  $\E_\Bbbk(n,r)$ such that $\rho(\bA) = \bB$. More precisely, for any
  fixed $1\le i \le n$ (respectively, $1 \le j \le n$), there exist
  $\bA(j)$ (resp., $\bA(i)$) in $\E_\Bbbk(n,r)^i_j$ such that $\bA =
  \bA(1) + \cdots + \bA(n)$. Call such an $\bA$ an $i$th block row
  (resp., $j$th block column) extension of $\bB$.
\end{property}

A priori, it is not clear that every extension can be constructed as a
block row or column extension. However, the next property guarantees
that every extension so arises.

\begin{property}[decomposition property]
  For any given $\bA$ in $\E_\Bbbk(n,r)$ and any fixed $1\le i \le n$
  (respectively, $1 \le j \le n$), there exist $\bA(j)$ (resp.,
  $\bA(i)$) in $\E_\Bbbk(n,r)^i_j$ such that $\bA = \bA(1) + \cdots +
  \bA(n)$. Call such a decomposition an $i$th block row (resp., $j$th
  block column) decomposition.
\end{property}

Property 1 for $(n,r)$ immediately implies Property 2 for $(n,r-1)$;
this follows from Remark \ref{rmk:block-sums}, because if $\bA$ is in
$\E_\Bbbk(n,r)$ and $\rho(\bA)=\bB$ then the sum of any chosen block
row or column of $\bA$ is equal to $\bB$.  It should be noted however
that we work in the opposite direction: we need Property 2 for
$(n,r-1)$ as an inductive hypothesis in order to prove Property 1 for
$(n,r)$.

Now we introduce free patterns, which index a set of entries in a
matrix which can be freely assigned to arbitary values in the ring
$\Bbbk$.  Free patterns correspond to a choice of free variables in
the solution of a consistent linear system.  Once the free pattern
entries have been assigned, the system has a unique solution.

\begin{property}[free patterns for extensions]
  For any fixed index $1\le i \le n$ (respectively, $1 \le j \le n$),
  there exists a subset $F(n,r)$ of $I'(n,r) \times I'(n,r)$, possibly
  empty and depending on the index, such that for any $\bB$ in
  $\E_\Bbbk(n,r-1)$ and any assignment $f: F(n,r) \to \Bbbk$ there
  exists a unique $\bA$ in $\E_\Bbbk(n,r)$ with $\rho(\bA) = \bB$.
\end{property}

\begin{rmk}
We often identify the elements of $F(n,r)$ with the matrix entries
they index. 
\end{rmk}

It follows from the symmetry in equations
\eqref{eq:slice-row}--\eqref{eq:stability} or from the existence of
the actions in \eqref{eq:W_n-actions} that if $F(n,r)$ is a given free
pattern, then by applying an arbitrary permutation $y\in W_n$ to all
its row or column indices, we obtain another free
pattern. Furthermore, by interchanging row and column indices in
$F(n,r)$, i.e., transposing, we obtain another free pattern.

\begin{property}[free patterns for decompositions]
  For any given $\bA$ in $\E_\Bbbk(n,r)$ and any fixed $1\le i \le n$
  (respectively, $1 \le j \le n$), there exists a subset $D(n,r)$ of
  $\{1, \dots, n\} \times I'(n,r) \times I'(n,r)$ such that for any
  given assignment $f: D(n,r) \to \Bbbk$ there exist unique $\bA(j)$
  (resp., $\bA(i)$) in $\E_\Bbbk(n,r)^i_j$ such that $\bA = \bA(1) +
  \cdots + \bA(n)$.
\end{property}

Properties 3, 4 are refinements of Properties 1, 2 (respectively) that
precisely quantify the amount of freedom in constructing solutions to
their underlying linear systems.

Suppose that $\bA$ is in $\E_\Bbbk(n,r)$. Fix $i$ and consider its
$i$th block row $\bA^i_* = (\bA^i_j)_{j=1}^n$. By Lemma
\ref{lem:blocks-are-special}, each block $\bA^i_j$ of $\bA^i_*$ gives
an invariant $\bB(j)$ in $\E_\Bbbk(n,r-1)^i_j$, and the correspondence
is given by
\begin{equation}\label{eq:label-iso}
  a^{i\, i_2 \cdots i_r}_{j\,j_2 \cdots j_r} = b(j)^{i_2 \cdots i_r}_{j_2 \cdots j_r}
\end{equation}
as $j$ runs from $1$ to $n$. We let $\pi^i$ be the bijection between
the set of pairs $(i\, i_2 \cdots i_r, j\,j_2 \cdots j_r)$ indexing
entries of $\bA^i_*$ on the left hand side of \eqref{eq:label-iso} and
the set of triples $(j, i_2 \cdots i_r, j_2 \cdots j_r)$ indexing
entries on the right hand side. Similarly, transposing rows and
columns and $i,j$ we obtain a bijection $\pi_j$ between the set of
pairs indexing entries of the $j$th block column $\bA^*_j$ and a
corresponding set of triples.  The following lemma is immediate from
Remark \ref{rmk:block-sums}.

\begin{lem}\label{lem:bijections}
Let $\bA \in \E_\Bbbk(n,r)$ and set $\bB = \rho(\bA) \in
\E_\Bbbk(n,r-1)$.  The map $\pi^i$ (resp., $\pi_j$) is a bijection
between a set of labels for the entries of $\bA^i_*$ (resp.,
$\bA^*_j$) and a set of labels for a solution $\bB = \bB(1)+ \cdots +
\bB(n)$ to the decomposition problem in degree $r-1$.
\end{lem}

Now suppose that $F(n,r)$ exists, in other words, that the free
extension pattern for constructing $\bA$ (from $\bB$) in Lemma
\ref{lem:bijections} exists (note that $F(n,r)$ is necessarily based
on a designated block row or column).  If $F(n,r)$ is based on the
$i$th block row (resp., $j$th block column) then we define
\begin{equation}\label{eq:F'-def}
  F'(n,r) = \{(i_1 \cdots i_r, j_1 \cdots j_r) \in F(n,r): i_1 = i\ 
  (\text{resp.},\ j_1 = j)\}.
\end{equation}
Furthermore, we define $F''(n,r)=F(n,r) \setminus F'(n,r)$, so that
\begin{equation}\label{eq:crux}
F(n,r) = F'(n,r) \sqcup F''(n,r).
\end{equation}
This disjoint decomposition is the crux of the interleaved induction
that will prove Properties 1--4.  Of great importance for our results
is the fact that the map $\pi^i$ (resp., $\pi_j$) in Lemma
\ref{lem:bijections} restricts to a bijection
\[
F'(n,r)\cong D(n,r-1),
\]
where the righthand-side is the free decomposition pattern for writing
$\bB = \bB(1)+ \cdots + \bB(n)$ in Lemma \ref{lem:bijections}.  In
particular, this means that Property 3 for $(n,r)$ immediately implies
Property 4 for $(n,r-1)$.  However, our induction proceeds in the
reverse direction, using the inverse of the above bijection to
construct $F'(n,r)$ from $D(n,r-1)$. Then we construct $F''(n,r)$ and
glue them together according to \eqref{eq:crux} in order to obtain
$F(n,r)$.  In this manner, we explicitly construct and interrelate the
free extension patterns and free decomposition patterns.

\begin{defn}
From now on we order $I(n,r)$ lexicographically, and we do the same
for the row and column indices of any matrix in
$\Mat_{I(n,r)}(\Bbbk)$.  A free pattern $F(n,r)$ is \emph{row-initial}
(resp., \emph{column-initial}) if, after identifying free-pattern
elements with their corresponding matrix entries, free-pattern entries
precede all other entries in each row (resp., column)
slice. Similarly, it is \emph{row-terminal} (resp.,
\emph{column-terminal}) if, after identifying free-pattern elements
with matrix entries, free-pattern entries come after all other entries
in each row (resp., column) slice.
\end{defn}

It is evident that free patterns $F(n,1)$ exist. Indeed, it is easily
checked that
\[
F(n,1) = \{(i,j): 2 \le i,j \le n\}
\]
is a row- and column-terminal pattern for $(n,1)$. To see this,
notice that once an assignment $f:F(n,1) \to \Bbbk$ has been chosen,
there is a unique extension $\bA$ in $\E_\Bbbk(n,1)$ of any given $b
\in \E_\Bbbk(n,0) = \Bbbk$ satisfying $a^i_j = f(i,j)$ for all $(i,j)
\in F(n,1)$. Note that $a^1_i$ and $a^i_1$ are then uniquely forced
for $ i>1 $. The first entry $a^1_1$ is forced in two ways, but
easily seen to be well-defined.

By applying appropriate permutations to the rows and/or columns of
$F(n,1)$ above, one gets row- and column-terminal, row-initial and
column-terminal, and row-terminal and column-initial free patterns for
$(n,1)$. Our proof of the four basic properties in the next section
will show that these variations of free patterns always exist, for any
$(n,r)$. Note that the distinguished element $w_0 \in W_n$ given by
$w_0(j) = n+1-j$ for $j = 1, \dots, n$ interchanges initial and
terminal patterns.

We conclude this section with the following algorithm,
which (as we show in Corollary \ref{cor:colouring}) determines a
terminal (respectively, initial) free pattern in a randomly chosen row
or column of an extension $\bA$ in $\E_\Bbbk(n,r)$ of a given $\bB$ in
$\E_\Bbbk(n,r-1)$, assuming that it is the first such row or column to
be completed. The algorithm is independent of Properties 1--4, and
independent of the inductive proof of those properties.

\begin{algorithm}\label{algo}
The $\alpha$-slices of $I'(n,r)=\{\bi \in I(n,r): \#(\bi)=r\}$ are the
subsets
\[
\{i_1 \cdots i_{\alpha-1}\, i\, i_{\alpha+1}\cdots i_r: i = 1, \dots,
n\}
\]
for $\alpha = 1, \dots, r$.  We start by listing the elements of
$I'(n,r)$ in lexicographical order.  We recursively assign a colour (0
or 1) to each element of $I'(n,r)$ as follows.

As long as uncoloured elements exist, we find the largest
(respectively, smallest) uncoloured element, colour it, and repeat.
Colouring an element consists of the following two steps:
\begin{enumerate}\renewcommand{\theenumi}{\roman{enumi}}
\item Examine all the element's slices. If the element is the only
  uncoloured element in one of its slices, we say it is forced, and
  colour it 0. Otherwise, colour the element 1 to indicate that it is
  free.
\item If colouring the current element forces any additional elements
  in any of its slices, then colour those elements as well. (This
  happens when there is just one remaining uncoloured element in the
  slice, after the current element is coloured.)
\end{enumerate}
We note that colouring is a recursive process, because of (ii). Let
$I'_1(n,r)$ be the set of elements of $I'(n,r)$ coloured 1 by the
above procedure.
\end{algorithm}

\begin{example}
By always choosing the largest uncoloured element, for the case of
$I'_1(5,2)$ the algorithm produces the following colouring:
\[
  I'_1(5,2): \quad
 \setlength{\arraycolsep}{2.7pt}\scriptsize
\begin{array}{c|cccc|cccc|cccc|cccc|cccc|}
  &\rt{12}&\rt{13}&\rt{14}&\rt{15}&\rt{21} &\rt{23}&\rt{24}&\rt{25}
  &\rt{31}&\rt{32} &\rt{34}&\rt{35}&\rt{41}&\rt{42}&\rt{43}  &\rt{45} 
    &\rt{51}&\rt{52}&\rt{53}  &\rt{54}   \\ \hline
       &0&0&0&0	&0&0&1&1	 &0&1&1&1	&0&1&1&1 &0&1&1&1
       \\ 
     \hline
\end{array}\; .
\]
Thus we have $I'_1(5,2) = \{54, 53, 52, 45, 43, 42, 35, 34, 32, 25, 24\}$.
\end{example}

Let $\bB$ in $\E_\Bbbk(n,r-1)$ be given, and fix some index $\bi$
(respectively, $\bj$) in $I'(n,r)$. Then any assignment
  \[
  f: \{ a^\bi_\bj : \bj \in I'_1(n,r)\} \to \Bbbk \qquad (resp., f: \{
  a^\bi_\bj : \bi \in I'_1(n,r)\} \to \Bbbk )
  \]
determines the $\bi$th row $a^\bi_*$ (resp., $\bj$th column $a^*_\bj$)
of a matrix $\bA$ satisfying equations \eqref{eq:slice-row} (resp.,
\eqref{eq:slice-col}) with respect to the given $\bB$.

\begin{rmk}\label{rmk:promise}
(i) In Corollary \ref{cor:colouring} we will show that that, under
suitable inductive hypotheses, there exists an extension $\bA$ of the
given $\bB$ which agrees with the row or column determined by the
above algorithm.

(ii) We prefer to work with terminal free patterns, because they are
compatible with restriction, in the sense that by excising all indices
containing an $n$ we obtain a free pattern for $n-1$. This preference
pervades all of the examples and some of the proofs in the next
section.
\end{rmk}

\section{Proof of Properties 1--4}\label{sec:proof}\noindent
We remind the reader that $\Bbbk$ is an arbitrary commutative
ring. Now we are ready to start the inductive proof of the four
properties. For each property, we need to assume an earlier instance
of one or more properties. We begin with Property 1.

%
%

\begin{prop}\label{prop:extend}
Assume Property 1 for $(n-1,r)$.  Then:
\begin{enumerate}
\item For any $i$, $j$ with $1 \le i,j \le n$ and any given $\bB$ in
  $\E_\Bbbk(n,r-1)^i_j$, there exists some $\bA$ in
  $\E_\Bbbk(n,r)^i_j$ such that $\rho(\bA) = \bB$.

\item
  \begin{enumerate}
  \item Fix some $j$ in $\{1, \dots, n\}$. Suppose given
    the data
  \[\bB(1), \dots, \bB(n)\] with $\bB(i)$ in
  $\E_\Bbbk(n,r-1)^i_j$ for $i = 1, \dots, n$. Then there exist
  corresponding $\bA(i)$ in $\E_\Bbbk(n,r)^i_j$ for $i = 1, \dots, n$
  such that $\rho(\bA(i)) = \bB(i)$ for all $i$. 

  \item Similarly, fix some $i$ in $\{1, \dots, n\}$. Suppose
    given the data
  \[\bB(1), \dots, \bB(n)\] with 
  $\bB(j)$ in
  $\E_\Bbbk(n,r-1)^i_j$ for $j = 1, \dots, n$. Then there exist
  corresponding $\bA(j)$ in $\E_\Bbbk(n,r)^i_j$ for $j = 1, \dots, n$
  such that $\rho(\bA(j)) = \bB(j)$ for all $j$. 

  \item In either case (i) or (ii), the sum $\bA(1) + \cdots + \bA(n)$
    extends $\bB(1) + \cdots + \bB(n)$.
  \end{enumerate}
  
\item If in addition Property 2 holds for $(n,r-1)$ then Property 1
  holds for $(n,r)$.
\end{enumerate}
\end{prop}

\begin{proof}
(a) This follows from Proposition \ref{prop:n-reduce-iso}, which
  reduces the question to the problem of extending from
  $\E_\Bbbk(n-1,r-1)$ to $\E_\Bbbk(n-1,r)$, which is solved by the
  hypothesis.

(b) is immediate from part (a) and the linearity of $\rho$.

(c) Let $\bB$ in $\E_\Bbbk(n,r-1)$ be given. By the decomposition
  property for $(n,r-1)$, we can find $\bB(j)$ in
  $\E_\Bbbk(n,r-1)^n_j$ for $j = 1, \dots, n$ such that $\bB = \bB(1)
  + \cdots + \bB(n)$. By part (b)(ii), there exist corresponding
  $\bA(j)$ in $\E_\Bbbk(n,r)^n_j$ such that $\rho(\bA(j)) = \bB(j)$
  for all $j = 1, \dots, n$. Put $\bA = \bA(1) + \cdots +
  \bA(n)$. Then by linearity of $\rho$ it follows that $\rho(\bA) =
  \bB$, as required. This shows that $\bA$ is a last block row
  extension of $\bB$. The proof for any other block row or column is
  similar.
\end{proof}

Having dealt with the existence of extensions, we now consider the
question of their uniqueness.

\begin{lem}\label{lem:rho-injects}
  Suppose that $n \le r$. Then
  \begin{enumerate}
  \item Extensions (if any) from $\E_\Bbbk(n,r-1)$ to $\E_\Bbbk(n,r)$
    are unique.
  \item Restriction $\rho: \E_\Bbbk(n,r) \to \E_\Bbbk(n,r-1)$ is
    injective.
  \item $\bzero$ is the only extension in $\E_\Bbbk(n,r)$ of $\bzero$.
  \end{enumerate}
\end{lem}

\begin{proof}
(a) Suppose that $\bA$ in $\E_\Bbbk(n,r)$ extends $\bB$ in
  $\E_\Bbbk(n,r-1)$. We use Proposition \ref{prop:inv-properties}. Let
  $\bi = i_1 \cdots i_r$, $\bj = j_1 \cdots j_r$ be in $I(n,r)$. If
  $\vt(\bi) \ne \vt(\bj)$ then by Proposition
  \ref{prop:inv-properties}(c), $a^\bi_\bj = 0$.

  So assume for the rest of the proof that $\vt(\bi) = \vt(\bj)$. If
  $\bi$ has a repeated value, then by Proposition
  \ref{prop:inv-properties}(d), $a^\bi_\bj = b^\bp_\bq$, where $\bp,
  \bq$ are obtained from $\bi, \bj$ by removing one of the duplicate
  values (from the same place). So such entries of $\bA$ are
  determined by $\bB$. If $r>n$, then all multi-indices in $I(n,r)$
  must have at least one duplicate value, so we are done in that case.

  We are left with the case $r=n$ and $\#(\bi) = n$ ($\bi$ has no
  repeated values). Then the same is true of $\bj$, and by Proposition
  \ref{prop:inv-properties}(b), with $\bp = i_1 \cdots i_{n-1}$ and
  $\bq = j_1 \cdots j_{n-1}$ we have
  \[
  \textstyle \sum a^{\bp\,*}_{\bq\,j_r} = b^\bp_\bq .
  \]
  Exactly $n-1$ of the possible $n$ values from $\{1, \dots, n\}$
  appear in $\bi$; similarly for $\bj$. Hence at most one term in the
  above sum can be non-zero, because of value-type, so $a^\bi_\bj =
  b^\bp_\bq$. So in this case, $\bA$ is also determined by $\bB$. This
  proves part (a).

(b) This follows from part (a). If $\bA \in \ker \rho$, then $\bA =
  \bzero$ by uniqueness.

(c) This follows from part (b).
\end{proof}

\begin{rmk}\label{rmk:sharper}
  Suppose that $n \le r+1$. If $\bC$ is any invariant in
  $\E_\Bbbk(n,r)$ with zero last block row, then $\bC = \bzero$. The
  same holds for any other block row or column. This follows from the
  proof of Lemma \ref{lem:rho-injects}(a); the assumption of a zero
  block row removes one degree of freedom from column slice sums, so
  the uniqueness conclusion is still valid, so $\bC=\bzero$.
\end{rmk}

The next result gives conditions under which we can construct
extensions with prescribed partial information in a chosen block row
or column. This is a crucial technical result needed to prove
Property 2. The chosen block row or column is controlled by its
free-pattern $F'(n,r) \cong D(n,r-1)$. We say that the prescribed
information is \emph{compatible} with the extension problem for a
given $\bB$ in $\E_\Bbbk(n,r-1)$ if there is some partial assignment
from a subset of some $F'(n,r)$ to $\Bbbk$ which gives the prescribed
information.

\begin{lem}\label{lem:partial}
Assume Property 1 for $(n-1,r)$ and Property 4 for $(n,r-1)$.  Suppose
that $\bB$ in $\E_\Bbbk(n,r-1)$ is given. Fix a choice of block row
(respectively, block column) and a choice of any number of compatible
prescribed rows (resp., columns) in the chosen block row (resp.,
column).  Then there exists an $\bA$ in $\E_\Bbbk(n,r)$ satisfying
$\rho(\bA)=\bB$ which agrees with the prescribed rows (resp.,
columns).
\end{lem}

\begin{proof}
Suppose we have fixed on an $i$th block row extension (the argument
for a block column extension is similar and left to the reader). By
hypothesis, a set $D(n,r-1)$ exists satisfying Property 4 for
$(n,r-1)$.  Any assignment to this set determines a decomposition of
the given $\bB$, say $\bB = \bB(1)+ \cdots + \bB(n)$ where each
$\bB(j) \in \E_\Bbbk(n,r-1)^i_j$.  We define $F'(n,r)$ to correspond
to $D(n,r-1)$ under the bijection $\pi^i$ of Lemma
\ref{lem:bijections}.  Assignments to $D(n,r-1)$ correspond to
assignments to $F'(n,r)$, which complete the $i$th block row of $\bA$
in a way that satisfies all relevant extension equations. Under this
correspondence, $\bA^i_j = \bB(j)$ for each $j=1,\dots,n$.  Once the
$i$th block row of $\bA$ is complete, we apply Proposition
\ref{prop:extend}(b) to find special extensions $\bA(j)$ in
$\E_\Bbbk(n,r)^i_j$ for all $j$ such that $\bA=\bA(1)+ \cdots +
\bA(n)$ extends $\bB$.

To finish, we simply note that by hypothesis the prescribed rows in
the $i$th block row are compatible with the extension problem for the
given $\bB$, so they are specified by a partial assignment $f'$ to
some subset of some $F'(n,r)$. We can extend $f'$ to an assignment
$f:F'(n,r)\to \Bbbk$ which then determines the $i$th block row of
$\bA$ as above, in a way that coincides with the prescribed
information.
\end{proof}

\begin{example}
We illustrate the above proof. Take $(n,r)=(4,2)$, and suppose that
row 42 of an extension has been prescribed, for a given $\bB$ in
$\E_\Bbbk(4,1)$. The prescribed row is part of the final block row, so
we construct a final block row extension.  The following
\[
F'(4,2): \quad
 \setlength{\arraycolsep}{2.7pt}\scriptsize
\begin{array}{c|ccc|ccc|ccc|ccc|}
  &\rt{12}&\rt{13}&\rt{14}&\rt{21} &\rt{23}&\rt{24}&\rt{31}&\rt{32}
  &\rt{34}&\rt{41}&\rt{42}&\rt{43}   \\ \hline 
    42  &&&		&&&\cm		&&\cm&\cm	&&\cm&\cm\\ 
   43    &&&		&&&\cm		&&\cm&\cm	&&\cm&\cm\\ 
     \hline
\end{array}\; 
\]
depicts a row- and column-terminal free-pattern $F'(4,2)$ for the last
block row of our desired extension $\bA$, where the checkmarked entries
correspond to elements of $F'(4,2)$, which can be freely assigned in
order to determine the last block row of a general extension. This
free pattern is calculated in Example \ref{Dn1} below.

The prescribed $42$-row merely determines the values of the five free
entries in that row. By assigning arbitrary values to the remaining
five entries in $F'(4,2)$ we complete the last block row of $\bA$, and
then complete $\bA$ by choosing special extensions of each block in
the last block row, and summing, as in the first paragraph of the
proof of Lemma \ref{lem:partial}.
\end{example}

We note the following immediate consequence of Lemma
\ref{lem:partial}, which was promised in Remark \ref{rmk:promise}.

\begin{cor}\label{cor:colouring}
Assume the same hypotheses as in Lemma \ref{lem:partial}. Suppose that
$(a^\bi_*)$ or $(a^*_\bj)$ is a row or column (where $\bi$ or $\bj \in
I'(n,r)$) determined by an assignment to the variables in that row or
column labelled by the set $I_1(n,r)$ in Algorithm \ref{algo}. Then
that row or column is a row or column of some extension $\bA$ of any
given $\bB$ in $\E_\Bbbk(n,r-1)$. 
\end{cor}

Now we are ready to prove Property 2. 

\begin{prop}\label{prop:decomp}
Assume Property 1 for $(n-1,r)$ and Property 4 for $(n-1,r-1)$. Then
Property 2 holds for $(n,r)$.
\end{prop}

\begin{proof}
Let $\bA$ in $\E_\Bbbk(n,r)$ be given. We choose to decompose $\bA$
based on its last block row. (The argument for any other block row or
column is similar.) By Proposition \ref{prop:n-reduce-iso}, the
problem of extending from $\E_\Bbbk(n,r-1)^n_j$ to $\E_\Bbbk(n,r)^n_j$
is equivalent to the problem of extending from $\E_\Bbbk(n-1,r-1)$ to
$\E_\Bbbk(n-1,r)$.  So by the first hypothesis and Proposition
\ref{prop:extend}(a),
\begin{equation}\label{eq:spec-ext}
\text{there exists $\bA(j)$ in $\E_\Bbbk(n,r)^n_j$ such that
  $\rho(\bA(j)) = \bA^n_j$}
\end{equation}
for all $j = 1, \dots, n$.  Any choice of $\bA(j)$ satisfying
\eqref{eq:spec-ext} makes the last block row of $\bA$ agree with that
of the sum $\bA(1)+\cdots +\bA(n)$, so that the last block row of the
difference $\bC=\bA-\bA(1)-\cdots -\bA(n)$ is zero. We need to show
that it is always possible to choose the $\bA(j)$ satisfying
\eqref{eq:spec-ext} in such a way that $\bC=\bzero$.

\smallskip\noindent\textbf{Case 1.}  If $n \le r+1$ then $\bC$ as
above is an invariant whose last block row is zero. The zero invariant
$\bzero$ is another such invariant. By Remark \ref{rmk:sharper}, it
follows that $\bC= \bzero$. This completes Case 1.

\smallskip\noindent
\textbf{Case 2.} Assume for the rest of the proof that $n > r+1$ (so
$n-r \ge 2$).  We aim to show that $\bA(j)$ for $j=2, \dots, n$ can be
chosen satisfying \eqref{eq:spec-ext} in such a way that all but the
last block of the first block column of $\bA - \bA(2) - \cdots -
\bA(n)$ is zero.  Working in \emph{reverse} order from right to left
along the last block row of $\bA$, we choose $\bA(j)$ for $j = n, n-1,
\dots, 2$ satisfying \eqref{eq:spec-ext} by the following process. For
each $j$, assuming that $\bA(n), \dots, \bA(j)$ have already been
chosen, we set $\bC(j)=\bA-\bA(n)-\cdots -\bA(j)$.

\smallskip
\emph{Step 1.} First we choose arbitrary $\bA(n), \dots, \bA(r+2)$
subject only to condition \eqref{eq:spec-ext}. This step is not
inductive.

\smallskip
\emph{Step 2.} We proceed by reverse induction on $j$ running from
$r+1$ down to $2$. At each stage, we claim that $\bA(j)$ satisfying
\eqref{eq:spec-ext} can be chosen so that:
\begin{equation}\label{eq:agrees}
  \begin{minipage}{4in}
  $\eta^n_j(\bA(j))$ agrees with $\eta^n_j(\bC(j+1))$ on all columns
    indexed by a label in $\Lin_{j-1}$
  \end{minipage}
\end{equation}
where $\Lin_j = \{j_1\cdots j_r \in I'(n,r): j_\alpha=\alpha \text{
  for all } \alpha=1, \dots, j\}$.  To see this, we apply Lemma
\ref{lem:partial} in the case $(n-1,r)$, since $\eta^n_j(\bC(j+1))$
belongs to $\E_\Bbbk(n-1,r)$. (Property 1 for $(n-1,r)$ implies the
existence of special extensions for $(n-1,r)$, which is equivalent to
Property 1 for $(n-2,r)$, so the hypotheses of Lemma \ref{lem:partial}
for the case $(n-1,r)$ are satisfied.) By the lemma, we can find
$\bA'(j)$ in $\E_\Bbbk(n-1,r)$ which agrees on the image under
$\eta^n_j$ of the columns of $\bC(j)$ indexed by $\Lin_{j-1}$. Then we
set $\bA(j) = \theta^n_j \bA'(j)$. This matrix satisfies
\eqref{eq:agrees}, so the claim is proved.

At this point we have inductively chosen $\bA(r+1)$ down to $\bA(2)$,
and we are ready for the final step, choosing $\bA(1)$.  Since
$\Lin_1$ indexes all the non-initialised columns in the first block
column, the only nonzero block in the first block column of $\bC(2)$
is $\bC(2)^n_1$. By construction, the same is true of the blocks in
the last block row.  Thus, we may apply Lemma
\ref{lem:special-criterion} to conclude that $\bC(2)$ belongs to
$\E_\Bbbk(n,r)^n_1$, and hence by setting $\bA(1) = \bC(2)$ we are
done.
\end{proof}

Step 2 above starts with the unique column in $\Lin_r = \{1 \cdots
r\}$, which isn't affected by any special invariant in
$\E_\Bbbk(n,r)^n_j$ for $j<r+1$. So the last opportunity to zero that
column is when we choose $\bA(r+1)$.  Similarly, as the induction in
Step 2 proceeds, controlled by the nested sequence
\[
\Lin_{r} \subset \Lin_{r-1} \subset \cdots \subset \Lin_1,
\]
the new columns that are zeroed in the running difference are
precisely those columns that cannot be affected in subsequent steps.

\begin{example}\label{ex:decomp42}
Assume that $\bA \in \E_\Bbbk(4,2)$. We illustrate Case 2 of the above
proof for a last block row decomposition, working from right to left
through the last block row.  Step 1 consists of subtracting an
arbitrary $\bA(4)$ in $\E_\Bbbk(4,2)^4_4$ such that
$\rho(\bA(4))=\bA^4_4$. Referring to the matrix forms depicted in
Examples \ref{ex:n23-r2-inv}, \ref{ex:special-inv} we see that after
Step 1,
\[
\setlength{\arraycolsep}{2.7pt}\scriptsize
\bC(4) = \bA - \bA(4) = 
\begin{array}{c|cccc|cccc|cccc|cccc|}
  &\rt{11}&\rt{12}&\rt{13}&\rt{14}&\rt{21}&\rt{22}&\rt{23}&\rt{24}&
  \rt{31}&\rt{32}&\rt{33}&\rt{34}&\rt{41}&\rt{42}&\rt{43}&\rt{44}\\ \hline
  11&*& & & & &*& & & & &*& & & & &*\\
  12& &t_1&*&*&*& &*&*&*&*& &*&*&*&* & \\
  13& &t_2&*&*&*& &*&*&*&*& &*&*&*&*& \\
  14& &*&s&0&*& &*&0&*&*& &0&*&*&*& \\  \hline
  21& &t'_1&*&*&*& &*&*&*&*& &*&*&*&*&  \\  
  22&*& & & & &*& & & & &*& & & & &*\\
  23& &t'_2&*&*&*& &*&*&*&*& &*&*&*&*& \\  
  24& &*&s'&0&*& &*&0&*&*& &0&*&*&*& \\ \hline  
  31& &t''_1&*&*&*& &*&*&*&*& &*&*&*&*& \\
  32& &t''_2&*&*&*& &*&*&*&*& &*&*&*&*& \\
  33&*& & & & &*& & & & &*& & & & &*\\  
  34& &*&s''&0&*& &*&0&*&*& &0&*&*&*& \\ \hline
\end{array} 
\]
where blank entries represent zero as usual. We omit showing the last
block row, where there are no choices. The nine explicit zeros shown
above are place-permutation symmetric to entries of $\bC(4)^4_4$,
which is $\bzero$ by construction. We must have $t_1+t_2 = s$,
$t'_1+t'_2 = s'$, and $t''_1+t''_2 = s''$ thanks to local GDS
conditions in the blocks.

The above equations imply that the image of column 12 under $\eta^4_3$
is consistent with that column of an extension of
$\eta^4_3(\bA^4_3)$. By Lemma \ref{lem:partial}, we can find an
$\bA(3)$ in $\E_\Bbbk(4,2)^4_3$ such that $\rho(\bA(3)) = \bC(4)^4_3$
and $\eta^4_3 (\bA(3))$ agrees with $\eta^4_3(\bC(4))$ in column
12. (Note that $\Lin_2=\{12\}$.) This implies that $\bC(3)$ has the form
\[
\setlength{\arraycolsep}{2.7pt}\scriptsize
\bC(3) = \bC(4) - \bA(3) = 
\begin{array}{c|cccc|cccc|cccc|cccc|}
  &\rt{11}&\rt{12}&\rt{13}&\rt{14}&\rt{21}&\rt{22}&\rt{23}&\rt{24}&
  \rt{31}&\rt{32}&\rt{33}&\rt{34}&\rt{41}&\rt{42}&\rt{43}&\rt{44}\\ \hline
  11&*& & & & &*& & & & &*& & & & &*\\
  12& &0&*&*&0& &*&*&*&*& &*&*&*&* & \\
  13& &0&*&*&0& &*&*&*&*& &*&*&*&*& \\
  14& &*&0&0&*& &0&0&*&*& &0&*&*&0& \\  \hline
  21& &0&*&*&0& &*&*&*&*& &*&*&*&*&  \\  
  22&*& & & & &*& & & & &*& & & & &*\\
  23& &0&*&*&0& &*&*&*&*& &*&*&*&*& \\  
  24& &*&0&0&*& &0&0&*&*& &0&*&*&0& \\ \hline  
  31& &0&*&*&0& &*&*&*&*& &*&*&*&*& \\
  32& &0&*&*&0& &*&*&*&*& &*&*&*&*& \\
  33&*& & & & &*& & & & &*& & & & &*\\  
  34& &*&0&0&*& &0&0&*&*& &0&*&*&0& \\ \hline
  \end{array} \; .
\]
The explicit zeros in column 21 must be zero because invariants are
place-permutation invariant, and column 21 is place-permutation
symmetric to column 12.

By Lemma \ref{lem:partial} there exists an $\bA(2)$ in
$\E_\Bbbk(4,2)^4_2$ such that $\rho(\bA(2)) = \bA^4_2$ and $\eta^4_2
(\bA(2))$ agrees with $\eta^4_2(\bC(3))$ in columns $\Lin_1 = \{1*\} =
\{12, 13, 14\}$. Hence the difference $\bC(2) = \bC(3) - \bA(2)$
satisfies the property
\[
\text{$\bC(2)^i_4 = \bzero$ and $\bC(2)^4_j = \bzero$ for all $i, j
  <4$.}
\]
By Lemma \ref{lem:special-criterion} it follows that $\bC(2)$ is
in $\E_\Bbbk(4,2)^4_1$, so by setting $\bA(1) = \bC(2)$ we obtain the
desired decomposition $\bA = \bA(1)+ \cdots + \bA(4)$. This completes
Example \ref{ex:decomp42}.
\end{example}

To prove Property 3 we need the following result, which describes how
to construct a free pattern for the special extension problem, for a
given pair $i,j$ of indices in the set $\{1, \dots, n\}$. We remind
the reader that, by Proposition \ref{prop:special-ext-form}(c), the
restriction of any $\bA$ in $\E_\Bbbk(n,r)^i_j$ belongs to
$\E_\Bbbk(n,r-1)^i_j$. Let $\theta^i_j$ be the isomorphism in
Proposition \ref{prop:n-reduce-iso}.

\begin{lem}\label{lem:free-special}
Assume Property 3 for $(n-1,r)$. Then $F(n,r)^i_j =
\theta^i_jF(n-1,r)$ is a free pattern for the problem of extending
invariants from $\E_\Bbbk(n,r-1)^i_j$ to $\E_\Bbbk(n,r)^i_j$, in the
sense that for any given $\bB$ in $\E_\Bbbk(n,r-1)^i_j$ and any
assignment $f^i_j: F(n,r)^i_j \to \Bbbk$, there is a unique special
invariant $\bA$ in $\E_\Bbbk(n,r)^i_j$ such that $\rho(\bA) = \bB$.
\end{lem}

\begin{proof}
Suppose some $\bB$ in $\E_\Bbbk(n,r-1)^i_j$ is given. To construct a
special extension $\bA$ in $\E_\Bbbk(n,r)^i_j$ we must set $\bA^i_j =
\bB$ and all other blocks in the $i$th block row and $j$th block
column to $\bzero$.  The rest of $\bA$ is determined by
place-permutation symmetry and the isomorphism $\theta^i_j:
\E_\Bbbk(n-1,r) \to \E_\Bbbk(n,r)^i_j$ of Proposition
\ref{prop:n-reduce-iso}, thus determined by assigning values to images
of the free variables in the free pattern $F(n-1,r)$.
\end{proof}

Now we are ready for the proof of Property 3.

\begin{prop}\label{prop:free-ptns-exist}
  Property 4 for $(n,r-1)$ and Property 3 for $(n-1,r)$ imply Property
  3 for $(n,r)$.
\end{prop}

\begin{proof} 
We choose to base the construction of $F(n,r)$ on the last block
row. The argument for any other block row or column is similar.  Let
$\pi = \pi^n$ be the map in Lemma \ref{lem:bijections}.  Set $F'(n,r)
= \pi^{-1} D(n,r-1)$. Then $F'(n,r)$ is a free pattern for completing
the last block row of an extension $\bA$ of $\bB$, because doing so is
equivalent to decomposing $\bB$ along its last block row. Let $\{
\bA^n_j: j=1, \dots, n \}$ be the last block row determined by some
chosen assignment $f': F'(n,r) \to \Bbbk$. Fix this choice of last
block row of $\bA$ for the rest of the argument.

To find the desired extension $\bA$, we apply Proposition
\ref{prop:extend}(a) to choose arbitrary special extensions $\bA(j)$
in $\E_\Bbbk(n,r)^n_j$ such that $\rho(\bA(j)) = \bA^n_j$ for all $j$,
and set $\bA = \bA(1) + \cdots + \bA(n)$.  By Lemma
\ref{lem:free-special}, there exists a free pattern $F(n,r)^n_j$ for
each $j$ and an assignment $f(j): F(n,r)^n_j \to \Bbbk$ that
determines $\bA^n_j$. Let $F''(n,r) = \bigcup_{j=1}^n F(n,r)^n_j$. To
continue, we introduce the notation
\[
\bA_{F''} := (a^\bi_\bj)_{(\bi,\bj) \in F''(n,r)}
\]
for the restriction of $\bA$ to $F''(n,r)$, and similarly for each
$\bA(j)$. We have $\bA_{F''} = \sum_j \bA(j)_{F''}$, so the given
assignments $f(j)$ induce a corresponding assignment $f'': F''(n,r)
\to \Bbbk$. For the given fixed last block row of $\bA$, it is clear
that there is an extension $\bA$ whose restriction to $F''(n,r)$
induces $f''$, for an arbitrary assignment $f'': F''(n,r) \to \Bbbk$,
because we can choose the $f(j)$ entries arbitrarily.

Every extension $\bA$ of $\bB$ with the specified last block row must
be of the form $\bA = \bA(1)+\cdots +\bA(n)$, where $\bA(j)$ is in
$\E_\Bbbk(n,r)^n_j$ and $\rho(\bA(j)) = \bA^n_j$ for all $j$. This
follows from the hypotheses, which by Proposition \ref{prop:decomp}
imply that the decomposition property holds for $(n,r)$.

To finish, we need to argue that $\bA$ (with the fixed last block row)
is uniquely determined by its assignment $f''$. To see this, suppose
that $\bA'$ is another extension of $\bB$, with the same last block
row, given by the same assignment $f''$. Then $\bA_{F''} =
\bA'_{F''}$. Thus the equations $\bA = \sum_j \bA(j)$, $\bA' = \sum_j
\bA'(j)$ imply by restriction that
\[
\textstyle \sum _j \bA(j)_{F''} = \sum_j \bA'(j)_{F''} \, .
\]
Each entry of $\bA(j)$, $\bA'(j)$ is uniquely expressible by the same
linear combination of the free pattern variables $a(j)^\bi_\bj$,
$a'(j)^\bi_\bj$ indexed by the free pattern $F(n,r)^n_j$. Hence, any
linear relations among the $\{\bA(j)\}$, $\{\bA'(j)\}$ are determined
by their restriction to $F''(n,r)$, which contains $F(n,r)^n_j$.  So
the above displayed equality implies that
\[
\textstyle \sum _j \bA(j) = \sum_j \bA'(j) \, .
\]
Hence $\bA = \bA'$, as required. This proves the desired uniqueness
statement. We have now shown that $F''(n,r)$ is a free pattern for
constructing an extension $\bA$ having the specified last block
row. Thus, the disjoint union $F(n,r) = F'(n,r) \sqcup F''(n,r)$ is a
free pattern for the extension problem (for the given $\bB$).
\end{proof}

We note the following immediate consequence of the above proof.

\begin{cor}
  Under the same hypotheses as the preceeding result based on an $i$th
  block row (respectively, $j$th block column) construction,
  \[
  F(n,r) = F'(n,r) \sqcup F''(n,r)
  \]
  where $F''(n,r) = \bigcup_{j=1}^n F(n,r)^i_j$ (resp., $F''(n,r) =
  \bigcup_{i=1}^n F(n,r)^i_j$) and $F'(n,r)$ is determined by the
  condition $\pi F'(n,r) = D(n,r-1)$.
\end{cor}

\begin{example}\label{ex:free-patterns}
(i) We now construct $F(4,2)$, under the assumption that $F(3,2)$ and
  $D(4,1) \cong F'(4,2)$ are known.  It is easy to check that $F(3,2)
  = \{(32,32)\}$. Hence
\[
F''(4,2)=\bigcup_{1\leq j \leq 4}\theta^j_4(F(3,2)) =\bigcup_{1\leq j
  \leq 4}\theta^j_4\{(32),(32)\}
\]
where $\theta^1_4\{32,32\}= \{32,43\}$, $\theta^2_4\{32,32\}=
\{32,43\}$, $\theta^3_4\{32,32\}= \{32,42\}$, and
$\theta^4_4\{32,32\}= \{32,32\}$.  We refer to Example \ref{Dn1} for
$F'(4,2)$. It follows that $F(4,2)$ is as depicted below:
 \[ F(4,2): \quad
 \setlength{\arraycolsep}{2.7pt}\scriptsize
\begin{array}{c|ccc|ccc|ccc|ccc|}
  &\rt{12}&\rt{13}&\rt{14}&\rt{21} &\rt{23}&\rt{24}&\rt{31}&\rt{32}
  &\rt{34}&\rt{41}&\rt{42}&\rt{43}   \\ \hline
  32   &&&	&&&	&&\cm&	&&\cm&\cm\\ \hline
  42  &&&	&&&\cm	&&\cm&\cm	&&\cm&\cm\\ 
 43    &&&	&&&\cm	&&\cm&\cm	&&\cm&\cm\\ 
     \hline
\end{array}\; .
\]

(ii) Now we consider $F(5,3)$, under the assumption that $F(4,3)$ and
$D(5,2)$ are both known. It is easy to check that $F(4,3) = \{(432,
432)\}$.  It follows that
 \[
  F'' (5,3): \quad
\setlength{\arraycolsep}{2.7pt}\scriptsize
\begin{array}{c|c|ccc|ccccc|ccccc|}
  &\rt{254}&\rt{325}&\rt{352}&\rt{354}%
  &\rt{425}&\rt{432}&\rt{435}&\rt{452}%
  &\rt{453}&\rt{524}&\rt{532}&\rt{534}%
  &\rt{542}&\rt{543}\\ \hline

  432 &  && &  &&\cm&& &  &&\cm&&\cm&\cm \\ 
    \hline
\end{array}\,
 \]
since $\theta^5_5(432,432)=(432,432)$,
$\theta^4_5(432,432)=(432,532)$, $\theta^3_5(432,432)=(432,542)$, and
$\theta^2_5(432,432)=(432,543)$.  See Example \ref{ex:D52} for
$F'(5,3)$. Taking the union of $F''(5,3)$ with $F'(5,3)$ gives
$F(5.3)$ as depicted below:
\[
  F(5,3): \quad
\setlength{\arraycolsep}{2.7pt}\scriptsize
\begin{array}{c|c|ccc|ccccc|ccccc|}
  &\rt{254}&\rt{325}&\rt{352}&\rt{354}%
  &\rt{425}&\rt{432}&\rt{435}&\rt{452}%
  &\rt{453}&\rt{524}&\rt{532}&\rt{534}%
  &\rt{542}&\rt{543}\\ \hline
  432 &  && &  &&\cm&& &  &&\cm&&\cm&\cm \\ \hline
  532 &\cm &&\cm&\cm &&\cm&&\cm&\cm &&\cm&&\cm&\cm \\ 
  
  542 &\cm &\cm&\cm&\cm &\cm&\cm&\cm&\cm&\cm &\cm&\cm&\cm&\cm&\cm \\
  543 &\cm &\cm&\cm&\cm &\cm&\cm&\cm&\cm&\cm &\cm&\cm&\cm&\cm&\cm \\ \hline
\end{array}\,.
\]
This ends Example \ref{ex:free-patterns}.
\end{example}

It remains to prove Property 4. The proof given below closely follows the
proof of Proposition \ref{prop:decomp}.

\begin{prop}\label{prop:property-4}
Property 3 for $(n-1,r)$ implies Property 4 for $(n,r)$.
\end{prop}

\begin{proof}
We need to prove the existence of the set $D(n,r)$ satisfying
Property~4. Property~3 for $(n-1,r)$ implies Property 4
for $(n-1,r-1)$ and also implies Property 1 for $(n-1,r)$, so the
hypotheses of Proposition \ref{prop:decomp} are satisfied, and hence
its conclusion holds. In other words, any given $\bA$ in
$\E_\Bbbk(n,r)$ has a decomposition based on a chosen block row or
column. We assume for concreteness that it is based on the last block
row, as in the proof of Proposition \ref{prop:decomp}. The argument
closely follows the proof of that proposition.

\smallskip\noindent\textbf{Case 1.}  If $n \le r+1$ then by Case 1 of
the proof of Proposition \ref{prop:decomp}, the desired decomposition
$\bA = \bA(1) + \cdots+ \bA(n)$ is unique. Hence $D(n,r)$ is empty in
this case.

\smallskip\noindent\textbf{Case 2.} Assume henceforth that $n > r+1$.
By Lemma \ref{lem:free-special} and the hypothesis, free patterns
\[F(n,r)^n_j = \theta^n_j F(n-1,r)\] are available for any
$j=1,\dots, n$. We may assume that $F(n-1,r)$ is row- and
column-terminal.

\smallskip\emph{Step 1.} For each $j=r+2, \dots, n$ we choose arbitrary
assignments $F(n,r)^n_j \to \Bbbk$. Each assignment uniquely
determines a matrix $\bA(j)$ in $\E_\Bbbk(n,r)_j^n$.  Therefore we
define
\begin{equation*}
D'(n,r) = \bigcup_{r+2 \le j \le n} \{j\}\times F(n,r)^n_j \, .
\end{equation*}
This set parametrises a set of free entries that uniquely determines
matrices $\bA(r+2), \dots, \bA(n)$ in Case 2 of the proof of
Proposition \ref{prop:decomp}. Set $\bC(j+2) = \bA - \sum_{j=r+2}^n
\bA(j)$.

\smallskip\emph{Step 2.}  Now we proceed by reverse induction on $j$
running from $r+1$ down to $2$, in order to identify a set $D''(n,r)$
of free entries parametrising the choice of $\bA(j+1), \dots, \bA(2)$
satisfying condition \eqref{eq:agrees} in the proof of Proposition
\ref{prop:decomp}. Clearly the union
\begin{equation*}\label{toomuch2}
\bigcup_{2\leq j \leq r+1} \{j\}\times F(n,r)^n_j 
\end{equation*}
is an upper bound on $D''(n,r)$. This bound is not tight due to linear
dependencies caused by the prescription of columns labelled by
elements of ${\Lin_{j-1}}^{\Sym_r}$ in the proof of Step 2 of
Proposition \ref{prop:decomp}. We have to remove those additional
dependencies in order to obtain $D''(n,r)$.

We do this by applying a modified version of Algorithm
\ref{algo}. Namely, for each fixed $j = r+1, \dots, 2$ we initialise
each entry $j_1 \cdots j_r$ of $I'(n,r)$ containing $j$ to colour 0,
and (following Step 2 in Case 2 of the proof of Proposition
\ref{prop:decomp}) do the same for each entry belonging to
${\Lin_{j-1}}^{\Sym_r}$. Then we run Algorithm \ref{algo} to determine
the independent (free) columns in a generic row of an extension. Let
$I'_j(n,r)$ be the set of elements coloured 1 in that algorithm. We
define
\[
\ov{F}(n,r)^n_j = \{(i_1\cdots i_r, j_1\cdots j_r) \in F(n,r)^n_j :
j_1\cdots j_r \in I'_j(n,r)\}.
\]
and we accordingly set
\[
D''(n,r) =  \bigcup_{2\leq j \leq r+1} \{j\}\times
  \ov{F}(n,r)^n_j \, .
\]
We claim that the desired pattern $D(n,r)$ is equal to the union
\[
D(n,r) = D'(n,r) \cup D''(n,r).
\]
To see this, observe that every decomposition $\bA = \bA(1) +\cdots+
\bA(n)$ in Proposition \ref{prop:decomp} determines a unique
assignment
\begin{equation}\label{eq:assign-D}
f: D(n,r) \to \Bbbk
\end{equation}
by setting $f(\{j\} \times (\bp,\bq)) = a(j)^\bp_\bq$.  Conversely,
for each $r+1 \ge j \ge 2$, each assignment to $\ov{F}(n,r)^n_j$ along
with the values in the prescribed columns indexed by
${\Lin_{j-1}}^{\Sym_r}$ forces a corresponding assignment to
$F(n,r)^n_j$. (Indeed, the algorithm was designed with that purpose in
mind.) It thus follows from the proof of Case 2 of Proposition
\ref{prop:decomp} that each assignment as in \eqref{eq:assign-D}
determines a unique decomposition of the form $\bA = \bA(1) +\cdots+
\bA(n)$ with the required properties.
\end{proof}

\begin{example}\label{Dn1}
To illustrate the above proof, we construct $D(n,1)$, assuming that
$F(n-1,1) = \{(i,j): 2 \le i,j \le n-1\}$. For each $j=1, \dots,
n$ we have
\[
F(n,1)^n_j = \{ (p,q): 2 \le p \le n-1,\ 2 \le q \le n,\ q \ne j\}.
\]
Furthermore, $\ov{F}(n,1)^n_2$ is obtained from $F(n,1)^n_2$ by
excising all entries in its leftmost column; that is,
\[
\ov{F}(n,1)^n_2 = \{ (p,q): 2 \le p \le n-1,\ 4 \le q \le n\}.
\]
Thus, as in the proof of Proposition \ref{prop:property-4}, we obtain
\[
D(n,1) = \{2\}\times \ov{F}(n,r)^n_2 \quad\cup\quad \bigcup_{3\leq j
  \leq n}\{j\}\times F(n,r)^n_j \, .
\]
For instance, when $n=4$ this can be depicted by the following table:
\[ D(4,1): \quad
 \setlength{\arraycolsep}{2.7pt}\scriptsize
\begin{array}{c|cccc||cccc||cccc||cccc|}
\hline
j=&\multicolumn{4}{c||}{1}	
&\multicolumn{4}{c||}{2}
&\multicolumn{4}{c||}{3}
&\multicolumn{4}{c|}{4}					\\ \hline
   &1&2&3&4 &1&2&3&4 &1&2&3&4 &1&2&3&4 \\ \hline
  1 &&&& &&&& &&&& &&&&\\
  2 &&&& &&&&\cm && \cm&&\cm && \cm&\cm&\\
  3 &&&& &&&&\cm && \cm&&\cm && \cm&\cm&\\ 
    4 &&&& &&&& &&&& &&&&\\ 
    \hline
\end{array}
\]
in which the indices labelling checkmarked positions correspond to
elements of $D(4,1)$. By Lemma \ref{lem:bijections}, it follows that
$F'(4,2)$ can be depicted by
\[
F'(4,2): \quad
 \setlength{\arraycolsep}{2.7pt}\scriptsize
\begin{array}{c|ccc|ccc|ccc|ccc|}
  &\rt{12}&\rt{13}&\rt{14}&\rt{21} &\rt{23}&\rt{24}&\rt{31}&\rt{32}
  &\rt{34}&\rt{41}&\rt{42}&\rt{43}   \\ \hline
  42  &&&	&&&\cm	&&\cm&\cm	&&\cm&\cm\\ 
  43    &&&	&&&\cm	&&\cm&\cm	&&\cm&\cm\\ 
  \hline
\end{array}\; .
\]
\end{example}

\begin{example}\label{ex:D52}
We now compute $D(5,2)$, which is in bijection with $F'(5,3)$.  We
first run Algorithm \ref{algo} to calculate
$I'_j(5,2)$ for $j=3,2$:
\[
  I'_3(5,2): \quad
 \setlength{\arraycolsep}{2.7pt}\scriptsize
\begin{array}{c|cccc|cccc|cccc|cccc|cccc|}
  &\rt{12}&\rt{13}&\rt{14}&\rt{15}&\rt{21} &\rt{23}&\rt{24}&\rt{25}
  &\rt{31}&\rt{32} &\rt{34}&\rt{35}&\rt{41}&\rt{42}&\rt{43}  &\rt{45} 
    &\rt{51}&\rt{52}&\rt{53}  &\rt{54}   \\ \hline
       &0&0&0	&0	&0&0&0	&1	 &0&0&0&0	&0&1&0&1 &0&1 &0&1
       \\ 
     \hline
\end{array}\;  
\]
\[
  I'_2(5,2): \quad
 \setlength{\arraycolsep}{2.7pt}\scriptsize
\begin{array}{c|cccc|cccc|cccc|cccc|cccc|}
    &\rt{12}&\rt{13}&\rt{14}&\rt{15}&\rt{21} &\rt{23}&\rt{24}&\rt{25}&\rt{31}&\rt{32} &\rt{34}&\rt{35}&\rt{41}&\rt{42}&\rt{43}  &\rt{45} 
    &\rt{51}&\rt{52}&\rt{53}  &\rt{54}   \\ \hline
       &0&0&0	&0	 &0&0&0&0&0&0&0	&1	&0&0&1&1	&0&0 &1&1
       \\ 
     \hline
\end{array}\;  
\]
To justify this, we go through the algorithm for $I'_3(5,2)$ a little
more slowly.  The initial colouring is
 \[  
 \setlength{\arraycolsep}{2.7pt}\scriptsize
\begin{array}{c|cccc|cccc|cccc|cccc|cccc|}
  &\rt{12}&\rt{13}&\rt{14}&\rt{15}&\rt{21} &\rt{23}&\rt{24}
  &\rt{25}&\rt{31}&\rt{32} &\rt{34}&\rt{35}&\rt{41}&\rt{42}&\rt{43}  &\rt{45} 
    &\rt{51}&\rt{52}&\rt{53}  &\rt{54}   \\ \hline 
       &0&0&	&	&0&0&&	&0	&0&0&0	&	&&0&	& &&0&	\\ 
     \hline
\end{array}\;  .
\] 
We now list the colourings obtained after three complete iterations
(doing both Steps 1 and 2) of the algorithm:
  \[  
 \setlength{\arraycolsep}{2.7pt}\scriptsize
\begin{array}{c|cccc|cccc|cccc|cccc|cccc|}
        &\rt{12}&\rt{13}&\rt{14}&\rt{15}&\rt{21} &\rt{23}&\rt{24}&\rt{25}&\rt{31}&\rt{32} &\rt{34}&\rt{35}&\rt{41}&\rt{42}&\rt{43}  &\rt{45} 
    &\rt{51}&\rt{52}&\rt{53}  &\rt{54}   \\ \hline 
       &0&0&		&	&0&0&&	&0	&0&0&0	&	&&0&	&		&&0&1			\\ 
     \hline
\end{array}\;  
\] 
 \vspace{1mm}
  \[  
 \setlength{\arraycolsep}{2.7pt}\scriptsize
\begin{array}{c|cccc|cccc|cccc|cccc|cccc|}
        &\rt{12}&\rt{13}&\rt{14}&\rt{15}&\rt{21} &\rt{23}&\rt{24}&\rt{25}&\rt{31}&\rt{32} &\rt{34}&\rt{35}&\rt{41}&\rt{42}&\rt{43}  &\rt{45} 
    &\rt{51}&\rt{52}&\rt{53}  &\rt{54}   \\ \hline 
       &0&0&		&	&0&0&&	&0	&0&0&0	&0	&0&0&0	&0&1&0&1			\\ 
     \hline
\end{array}\;  
\]  \vspace{1mm}
   \[  
 \setlength{\arraycolsep}{2.7pt}\scriptsize
\begin{array}{c|cccc|cccc|cccc|cccc|cccc|}
        &\rt{12}&\rt{13}&\rt{14}&\rt{15}&\rt{21} &\rt{23}&\rt{24}&\rt{25}&\rt{31}&\rt{32} &\rt{34}&\rt{35}&\rt{41}&\rt{42}&\rt{43}  &\rt{45} 
    &\rt{51}&\rt{52}&\rt{53}  &\rt{54}   \\ \hline 
       &0&0&0		&0	&0&0&0&1	&0	&0&0&0	&0	&0&0&0	&0&1&0&1			\\ 
     \hline
\end{array}\;  
\]
respectively.  Thus we  obtain 
\[ \ov{F}(5,2)^5_3: \quad
  \setlength{\arraycolsep}{2.7pt}\scriptsize
\begin{array}{c|cccc|cccc|cccc|cccc|cccc|}
        &\rt{12}&\rt{13}&\rt{14}&\rt{15}&\rt{21} &\rt{23}&\rt{24}&\rt{25}&\rt{31}&\rt{32} &\rt{34}&\rt{35}&\rt{41}&\rt{42}&\rt{43}  &\rt{45} 
    &\rt{51}&\rt{52}&\rt{53}  &\rt{54}   \\ \hline 
  32   	&&&&			&&&&			&& &	&&& &	&		&&\cm&&\cm							\\ \hline
    42  	&&&&			&&&&\cm			&& && && &	&		&&\cm&&\cm				\\ 
   43    	&&&&			&&&&\cm			&& && && &	&		&&\cm&&\cm							\\ 
     \hline
\end{array}\; .
\]
Now let $j=2$.  The initial colouring is
 \[  \setlength{\arraycolsep}{2.7pt}\scriptsize
\begin{array}{c|cccc|cccc|cccc|cccc|cccc|}
        &\rt{12}&\rt{13}&\rt{14}&\rt{15}&\rt{21} &\rt{23}&\rt{24}&\rt{25}&\rt{31}&\rt{32} &\rt{34}&\rt{35}&\rt{41}&\rt{42}&\rt{43}  &\rt{45} 
    &\rt{51}&\rt{52}&\rt{53}  &\rt{54}   \\ \hline 
       &0&0&0&0	
       &0&0&0&0	
       &0&0&&						 
       &0&0&&
              &0&0&&    	\\   \hline
\end{array}\;  
\] 
In Step 1 of the algorithm, we 1-colour 54.  In Step 2 this forces us
to 0-colour 53; this in turn forces us to 0-colour 43; this in turn
forces us to 0-colour 45; this in turn forces us to 0-colour 35; this
in turn forces us to 0-colour 34.  Thus $I'_{2}(5,2)=\{54\}$.
Therefore we have
\[ \ov{F}(5,2)^5_2: \quad
  \setlength{\arraycolsep}{2.7pt}\scriptsize
\begin{array}{c|cccc|cccc|cccc|cccc|cccc|}
        &\rt{12}&\rt{13}&\rt{14}&\rt{15}&\rt{21} &\rt{23}&\rt{24}&\rt{25}&\rt{31}&\rt{32} &\rt{34}&\rt{35}&\rt{41}&\rt{42}&\rt{43}  &\rt{45} 
    &\rt{51}&\rt{52}&\rt{53}  &\rt{54}   \\ \hline 
  32   	&&&&			&&&&			&& &	&&& &	&		&&&&\cm							\\ \hline
    42  	&&&&			&&&&			&& && && &	&		&&&&\cm				\\ 
   43    	&&&&			&&&&			&& && && &	&		&&&&	\cm						\\ 
     \hline
\end{array}\; .
\]
We summarise this in the following table, where the last two blocks
are determined by $F(4,2)$ in Example \ref{ex:free-patterns}(i):
\[
D(5,2): \quad
\setlength{\arraycolsep}{2.7pt}\scriptsize
\begin{array}{c|c||ccc||ccccc||ccccc|}
\hline j=&{2}	
&\multicolumn{3}{c||}{3}
&\multicolumn{5}{c||}{4}
&\multicolumn{5}{c|}{5}					\\ \hline 
 &\rt{54}&\rt{25}&\rt{52}&\rt{54}%
  &\rt{25}&\rt{32}&\rt{35}&\rt{52}%
  &\rt{53}&\rt{24 }&\rt{32 }&\rt{34 }%
  &\rt{42}&\rt{43}\\ 
  \hline
   32 &\cm &&\cm&\cm &&\cm&&\cm&\cm &&\cm&&\cm&\cm \\ 
  
   42 &\cm &\cm&\cm&\cm &\cm&\cm&\cm&\cm&\cm &\cm&\cm&\cm&\cm&\cm \\
   43 &\cm &\cm&\cm&\cm &\cm&\cm&\cm&\cm&\cm &\cm&\cm&\cm&\cm&\cm \\ \hline
\end{array}\,
\]
in which we omit columns that have no information, in order to save space.
We hence obtain 
\[
  F' (5,3): \quad
\setlength{\arraycolsep}{2.7pt}\scriptsize
\begin{array}{c|c|ccc|ccccc|ccccc|}
  &\rt{254}&\rt{325}&\rt{352}&\rt{354}%
  &\rt{425}&\rt{432}&\rt{435}&\rt{452}%
  &\rt{453}&\rt{524}&\rt{532}&\rt{534}%
  &\rt{542}&\rt{543}\\ \hline
   532 &\cm &&\cm&\cm &&\cm&&\cm&\cm &&\cm&&\cm&\cm \\ 
  
  542 &\cm &\cm&\cm&\cm &\cm&\cm&\cm&\cm&\cm &\cm&\cm&\cm&\cm&\cm \\
  543 &\cm &\cm&\cm&\cm &\cm&\cm&\cm&\cm&\cm &\cm&\cm&\cm&\cm&\cm \\ \hline
\end{array}\,.
\] 
\end{example}

\begin{thm}\label{thm:properties}
  All the Properties 1--4 hold for any $n \ge 2$, $r\ge 1$. 
\end{thm}

\begin{proof}
This follows from the results of this section by a double induction on
$n,r$. To be precise, we summarize the results proved in Propositions
\ref{prop:extend}, \ref{prop:decomp}, \ref{prop:free-ptns-exist}, and
\ref{prop:property-4} in the table below:
\begin{center}
\begin{tabular}{c|l|c}
  Result & \multicolumn{1}{c|}{Hypotheses} & Conclusion \\ \hline
  \ref{prop:extend}& Prop.~1$(n-1,r)$, Prop.~2$(n,r-1)$ & Prop.~1$(n,r)$ \\
  \ref{prop:decomp}& Prop.~1$(n-1,r)$, Prop.~4$(n-1,r-1)$ & Prop.~2$(n,r)$ \\
  \ref{prop:free-ptns-exist}& Prop.~3$(n-1,r)$, Prop.~4$(n,r-1)$ & Prop.~3$(n,r)$ \\
  \ref{prop:property-4}& Prop.~3$(n-1,r)$, Prop.~4$(n-1,r-1)$ & Prop.~4$(n,r)$\\ \hline
\end{tabular}\; .
\end{center}
Properties 1 and 3 for $(n,1)$ are evident, for any
$n$. Properties 2 and 4 hold vacuously for $r=0$, for any $n$, as
there is nothing to do since $\E_\Bbbk(n,0) \cong \Bbbk$. These serve
as base cases for the induction.
\end{proof}

\appendix
\section{Gibson's theorem}\label{sec:Gibson}\noindent
We explain the connection to earlier work of P.M.~Gibson
\cite{Gibson}.  Assume in this appendix that $n > 1$ and that $\Bbbk$
is a (not necessarily commutative) unital ring. Gibson observed that
the algebra $\E_\Bbbk(n,1)$ of $n \times n$ GDS matrices is free over
the ring $\Bbbk$ and spanned by permutation matrices, and he
constructed an explicit basis of $\E_\Bbbk(n,1)$ of permutation
matrices.

In order to describe Gibson's basis, we define $i+1 \text{ mod } n$ to
be the unique element $t$ of $\{1, \dots, n\}$ such that $i \equiv t$
modulo $n$.  Let
\[
\bQ_n = \big( \delta_{i,\, j+1 \text{ mod } n} \big)_{i,j = 1, \dots, n}
\]
be the $n \times n$ circulant permutation matrix representing the
descending $n$-cycle $(n, n-1, \dots, 1) \in W_n$.  For example, if
$n=4$ we have
\[
\bQ_4 = 
\begin{bmatrix}
  0&1&0&0\\
  0&0&1&0\\
  0&0&0&1\\
  1&0&0&0
\end{bmatrix}.
\]
Let $m^i_j$ be the $(i,j)$-entry of $\bQ_n + \bI_n$, where $\bI_n =
\big( \delta_{i,j} \big)_{i,j = 1, \dots, n}$ is the $n \times n$
identity matrix.  Clearly
\[ m^i_{j} = 
\begin{cases}
  1 & \text{ if } i=j \text{ or } i+1 \text{ mod } n = j\\
  0 & \text{ otherwise.}
\end{cases}
\]
There are two zero entries in each column of $\bQ_n + \bI_n = \big(
m^i_{j} \big)$, one for each pair $(r,c) \in \Gamma_n$, where we define
$\Gamma_n$ to be the set of $(r,c)$ such that $r = c$ or $r+1 \text{
  mod } n = c$.  So there are $n(n-2)$ zero entries in $\bQ_n +
\bI_n$. If $m^r_c = 0$, there is a unique $n \times n$ permutation
matrix $\bG_{r,c} = \big(g^i_j\big)$ such that
\[
g^r_c = 1, \text{ and } g^i_j \le m^i_j \text{ for all } (i,j) \ne (r,c).
\]
In fact, one can show that the $(n-1) \times (n-1)$ submatrix obtained
by removing row $r$ and column $c$ of $\bG_{r,c}$ is equal to
\[
  \begin{cases}
    \bQ_{n-1} & \text{ if } r<c\\
    \bI_{n-1} & \text{ if } c<r.
  \end{cases}
\]
This provides a recursive description of $\bG_{r,c}$.  We obtain
$n(n-2)$ linearly independent permutation matrices $\bG_{r,c}$ in this
way, one for each $(r,c) \in \Gamma_n$.  Gibson proved the following
result.

\begin{thm}[{\cite{Gibson}*{Theorem~2.1}}]
Let $\Bbbk$ be any unital ring. Then the set of permutation matrices 
\[
  \{ \bG_{r,c}: 1 \le r,c \le n \text{ and } (r,c) \in \Gamma_n \}
  \cup \{ \bQ_n, \bI_n \}
\]
is a basis over $\Bbbk$ of $\E_\Bbbk(n,1)$. In particular, the algebra
of all $n \times n$ GDS matrices is free over $\Bbbk$ of rank $(n-1)^2
+ 1$.
\end{thm}

To express a given $n \times n$ GDS matrix $\bA = \big(a^i_j \big)$ as
a linear combination of permutation matrices, one sets
\begin{equation}
\bB = \big(b^i_j \big) = \bA - \textstyle \sum_{(r,c) \in \Gamma_n}
a^r_c \, \bG_{r,c}.
\end{equation}
Then it is easy to see that $\bB - b^n_1 \bQ_n - b^n_n \bI_n = \bzero$, so
\begin{equation}
\bA = b^n_1 \bQ_n + b^n_n \bI_n + \textstyle \sum_{(r,c) \in \Gamma_n}
a^r_c \, \bG_{r,c}
\end{equation}
is the desired linear combination.  This shows that the proposed basis
spans $\E_\Bbbk(n,1)$. One easily checks that it is linearly
independent.

Johnsen \cite{Johnsen} found a different basis of permutation matrices
for $\E_\Bbbk(n,1)$ under the assumption that $\Bbbk$ is a field. In
that case it suffices to focus on the doubly stochastic matrices (with
row and column sums equal to $1$). See also \cite{Lai} for related work.

\begin{rmk}
  It follows from Gibson's theorem that $\E_\Bbbk(n,1)$ is spanned by
  the set of $n \times n$ permutation matrices.  This is the $r=1$
  case of Theorem~\ref{thm:main1}. 
\end{rmk}


\end{document}